\documentclass[letterpaper , 12pt ,reqno]{article}

\usepackage{fullpage}		% MAKES IT HAVE SMALLER MARGINS
\usepackage{amsmath, amsthm}
\usepackage{eucal}
\usepackage{mathtools}
\usepackage{verbatim}
\usepackage{amssymb}
\usepackage{latexsym}
\usepackage{graphicx}
\usepackage{amsfonts}
\usepackage{cases}

\usepackage{bbding}

\allowdisplaybreaks     %This command allows very long align environments and lists to continue even past a page break.

\usepackage[dvipsnames]{xcolor}
\usepackage{colortbl}	%Allows you to put color in the tables rows, columns, or cells.

\usepackage{url}

\usepackage{float}		%Allows you to use \begin{figure}[H] command

\usepackage{tikz}			%This lets you make nice diagrams such as graphs, cartesian coordinate plane,  bubble, etc.
\usetikzlibrary{cd,positioning, shapes, arrows}		%This loads the commutative diagram commands for tikz, as well as wwts bubble
\usepackage{tkz-graph}

\usepackage{mdframed}		%Allows you to put text and math in a framed box

\usepackage{empheq}
\usepackage[most]{tcolorbox}

\newtcbox{\mymath}[1][]{%
    nobeforeafter, math upper, tcbox raise base,
    enhanced, colframe=blue!30!black,
    colback=blue!30, boxrule=1pt,
    #1}

\usepackage{subfigure}	%Use this to have figures placed side by side!!

\usepackage{soul}		%Allows you to highlight text using \hl{Blah}

\usepackage[autolanguage]{numprint}  %For proper spacing for comma as a decimal separtor (e.g., 1,328). For example $\numprint{119 740 619}$ will put commas in appropriate places and adjust spacing accordingly.

\usepackage{bm}		%Allows better math bold face than $\mathbf$. To use just type $\bm{math stuff here}$.

\usepackage{centernot}  % to make the ndivides command

\usepackage[toc]{appendix}	%To have an appendix at the end

\usepackage{chngcntr}		%Allows you to change the counter for figues in the appendix to take into account the section. For example, Figure A.2 is 2nd figure in first appendix. After \appendix put \counterwithin{figure}{section}

\usepackage{mathrsfs}		%Allows you to use mathscript font $\mathscr{Blah}$

\usepackage{enumerate}		%Allows you to use \begin{enumerate}[(i)] which denotes each bullet point as (i), (ii), etc.

\usepackage{wasysym}        %Allows you to use \smiley command to make smiley face

\usepackage{cancel}         %Allows you to draw a diagonal slash line to cancel something. Just type \cancel{blah}

\definecolor{shadecolor}{gray}{0.90}				
\def\boitegrise#1#2{\begin{centerline}{\fcolorbox{black}{shadecolor}{~
    \begin{minipage}[t]{#2}{\vphantom{~}#1\vphantom{$A_{\displaystyle{A_A}}$}}
            \end{minipage}~}}\end{centerline}\medskip}

\usepackage[shortlabels]{enumitem}	%This along with enumerate package above allows you to change the enumerate labels to the TiKz circled numbers labels for each \item

\setenumerate[1]{label=\thesubsection.\arabic*.}
\setenumerate[2]{label*=\arabic*.}

\usepackage{multicol}

\usepackage{setspace} %To reduce spacing in TOC
\usepackage{tablefootnote}

\usepackage[colorlinks=true, citecolor=blue, linktocpage]{hyperref}

\usepackage{cellspace}

\newcommand{\FibSeq}{\left(F_n\right)_{n \geq 0}}
\newcommand{\LucSeq}{\left(L_n\right)_{n \geq 0}}
\newcommand{\GibSeq}{\left(G_n\right)_{n \geq 0}}
\newcommand{\PellSeq}{\left(P_n\right)_{n \geq 0}}
\newcommand{\QellSeq}{\left(Q_n\right)_{n \geq 0}}
\newcommand{\BalSeq}{\left(B_n\right)_{n \geq 0}}

\newcommand{\LucBalSeq}{\left(C_n\right)_{n \geq 0}}

\newcommand{\JacSeq}{\left(J_n\right)_{n \geq 0}}
\newcommand{\JacLucSeq}{\left(j_n\right)_{n \geq 0}}

\newcommand{\DLB}{\mathfrak{C}}
\newcommand{\Mer}{\mathcalb{M}}
\newcommand{\MerLuc}{\mathcalb{F}}

\newcommand{\DoubleLucBalSeq}{\left(\DLB_n\right)_{n \geq 0}}
\newcommand{\MerSeq}{\left(\Mer_n\right)_{n \geq 0}}
\newcommand{\MerLucSeq}{\left(\MerLuc_n\right)_{n \geq 0}}

\newcommand{\LucasFormalUSeq}{\left(U_n(p,q)\right)_{n \geq 0}}
\newcommand{\LucasFormalVSeq}{\left(V_n(p,q)\right)_{n \geq 0}}

\newcommand{\LucasFormalUTerm}{U_n(p,q)}
\newcommand{\LucasFormalVTerm}{V_n(p,q)}
\newcommand{\LucasUSeq}{\left(U_n\right)_{n \geq 0}}
\newcommand{\LucasVSeq}{\left(V_n\right)_{n \geq 0}}

\newcommand{\GenericSeq}{\left(S_n\right)_{n \geq 0}}

\newcommand{\seqnum}[1]{\href{https://oeis.org/#1}{\rm \underline{#1}}}

\DeclareFontFamily{U}{BOONDOX-calo}{\skewchar\font=45 }
\DeclareFontShape{U}{BOONDOX-calo}{m}{n}{
  <-> s*[1.05] BOONDOX-r-calo}{}
\DeclareFontShape{U}{BOONDOX-calo}{b}{n}{
  <-> s*[1.05] BOONDOX-b-calo}{}
\DeclareMathAlphabet{\mathcalb}{U}{BOONDOX-calo}{m}{n}
\SetMathAlphabet{\mathcalb}{bold}{U}{BOONDOX-calo}{b}{n}
\DeclareMathAlphabet{\mathbcalb}{U}{BOONDOX-calo}{b}{n}

\newcommand{\red}[1]{\textcolor{red}{#1}}
\newcommand{\blue}[1]{\textcolor{blue}{#1}}
\newcommand{\redbf}[1]{\textcolor{red}{\textbf{#1}}}

\newtheorem{theorem}{Theorem}[section] %[section] here insures the start of each section has sets the theorem counter back to 1. 
\newtheorem{lemma}[theorem]{Lemma}
\newtheorem{corollary}[theorem]{Corollary}

\newtheorem*{conjecture*}{Conjecture}

\newtheorem{proposition}[theorem]{Proposition}

\theoremstyle{definition}

\theoremstyle{definition}
\newtheorem{definition}[theorem]{Definition}

\theoremstyle{plain}

\theoremstyle{definition}
\newtheorem{question}[theorem]{Question}

\theoremstyle{definition}
\newtheorem{remark}[theorem]{Remark}

\theoremstyle{definition}
\newtheorem{convention}[theorem]{Convention}

\numberwithin{equation}{section}

\begin{document}

\title{\redbf{Weighted sums of Lucas numbers of the first and second kind}}

\author{
{\large aBa~Mbirika}\\
{\small University of Wisconsin-Eau Claire}\\
\href{mailto:mbirika@uwec.edu}{\small\nolinkurl{mbirika@uwec.edu}}}

\date{}

\makeatletter
\newcommand{\subjclass}[2][2020]{%
  \let\@oldtitle\@title%
  \gdef\@title{\@oldtitle\footnotetext{#1 \emph{Mathematics Subject Classification.} #2}}%
}
\newcommand{\keywords}[1]{%
  \let\@@oldtitle\@title%
  \gdef\@title{\@@oldtitle\footnotetext{\emph{Key words and phrases.} #1.}}%
}
\makeatother


\maketitle

%%%%%%%%%%%%%%%%%%%%%%%%%%%%%%%%%%%%%%%
%%%%%%%%%%%%%%%%%%%%%%%%%%%%%%%%%%%%%%%
%%  REMEMBER TO REMOVE THIS \vspace  %%
%%%%%%%%%%%%%%%%%%%%%%%%%%%%%%%%%%%%%%%
\vspace{-.25in}

\begin{abstract}
In the \textit{Fibonacci Quarterly} in 1964, C.~R.~Wall gave the following weighted sum of generalized Fibonacci numbers: $\sum_{i=1}^n i G_i = n G_{n+2} - G_{n+3} + G_3$,
where $\left(G_n\right)_{n \geq 0}$ is defined by the recurrence $G_n = G_{n-1} + G_{n-2}$ with fixed $G_0, G_1 \in \mathbb{Z}$. In this paper, we generalize Wall's result to the Lucas sequences of the first and second kind, $\left(U_n(p,q)\right)_{n \geq 0}$ and $\left(V_n(p,q)\right)_{n \geq 0}$, and give closed forms for $\sum_{i=1}^n i U_i$ and $\sum_{i=1}^n i V_i$ by using Abel's summation by parts method. Moreover, we provide concrete applications, not only recovering the known weighted sums $\sum_{i=1}^n i F_i$ and $\sum_{i=1}^n i L_i$ of Fibonacci and Lucas numbers, respectively, but also add new identities to the literature for eight well-known sequences. In particular, we give closed forms for weighted sums of Lucas sequences of the first kind such as the Pell, balancing, Jacobstahl, and Mersenne numbers, and also Lucas sequences of the second kind such as the companion Pell, double Lucas-balancing, Jacobstahl-Lucas, and Mersenne-Lucas numbers.
\end{abstract}

%\newpage
\renewcommand{\baselinestretch}{0.75}\normalsize
{\footnotesize\tableofcontents}
\renewcommand{\baselinestretch}{1.0}\normalsize

%%%%%%%%%%%%%%%%%%%%%%%%%%%%%%%%%%%%%
%%%%%%%%%%%%%%%%%%%%%%%%%%%%%%%%%%%%%
%%%%%%%%%    SECTION 1    %%%%%%%%%%%
%%%%%%%%%%%%%%%%%%%%%%%%%%%%%%%%%%%%%
%%%%%%%%%%%%%%%%%%%%%%%%%%%%%%%%%%%%%

%\newpage

\section{Introduction, history, and summary of main results}\label{sec:motivation}

For a sequence $\GenericSeq$, there is much in the literature about closed forms for weighted sums of the form $\sum_{i=1}^n \omega_i S_i$ where $\omega_i \in \mathbb{Z}$. For the Fibonacci sequence, this has been well studied for a variety of weights. We provide some of these results in Table~\ref{table:weighted_Fibonacci_sum_examples}.

\begin{table}[!ht]
\renewcommand{\arraystretch}{1.1}
\centering
\begin{tabular}{|c|c|c|}
\hline
\red{$\omega_i$} & \red{Closed form for $\sum_{i=1}^n \omega_i F_i$} & \red{Source}\\ \hline\hline
$1$ & $F_{n+2} - 1$ & Lucas 1878~\cite[(56)\footnotemark]{Lucas1878}\\ \hline
$(-1)^{i+1} \binom{n}{i}$ & $F_n$ & Ruggles 1963~\cite{Ruggles1963}\\ \hline
$\binom{n}{i}$ & $F_{2n}$ & Ruggles 1963~\cite{Ruggles1963}\\
&  & and Vajda 1989~\cite[(47)\footnotemark]{Vajda1989}\\
\hline
$i$ & $(n+1) F_{n+2} - F_{n+4} + 2$ & Gauthier 1995~\cite{Gauthier1995}\\ \hline
$i^2$ & $(n + 1)^2 F_{n + 2} - (2 n + 3) F_{n + 4} + 2 F_{n + 6} - 8$ & Gauthier 1995~\cite{Gauthier1995}\\ \hline
$i^3$ & $(n + 1)^3 F_{n + 2} - (3n^2 +9n + 7) F_{n + 4}$ & Gauthier 1995~\cite{Gauthier1995}\\
& $+ (6n + 12) F_{n + 6} - 6 F_{n+8} + 50$ & \\ \hline
$2^{n-i}$ & $2^n - F_{n+2}$ & Benjamin-Quinn 2003~\cite[(10)\footnotemark]{Benjamin-Quinn2003} \\ \hline
$(-1)^{i}$ & $(-1)^n F_{n-1} - 1$ & Benjamin-Quinn 2003~\cite[(21)\footnotemark]{Benjamin-Quinn2003} \\ \hline
\end{tabular}
\caption{Closed forms for a variety of weighted Fibonacci sums}
\label{table:weighted_Fibonacci_sum_examples}
\end{table}
\footnotetext[1]{This result is a consequence of Lucas' Identity~(56), which unfortunately had two typos, a minor typo for $\sum_{i=1}^n U_{ir}$ and a nontrivial typo for $\sum_{i=1}^n V_{ir}$, in his seminal 1878 paper (see Subsection~\ref{subsec:consecutive_sum_when_p_minus_q_does_not_equal_1}). Moreover, see Theorem~\ref{thm:corrected_Identity_(56)_for_V_seq} in the appendix for our fixed version and proof of the closed form for $\sum_{i=1}^n V_{ir}$.}
\footnotetext[2]{Vajda actually gives a closed form for $\sum_{i=1}^n \binom{n}{i} G_i$ where $\GibSeq$ is a generalized Fibonacci sequence.}
\footnotetext[3]{This is an equivalent formulation of Benjamin and Quinn's Identity~(10).}
\footnotetext[4]{This is an equivalent formulation of Benjamin and Quinn's Identity~(21).}

The methods to prove the closed forms in Table~\ref{table:weighted_Fibonacci_sum_examples} vary greatly among the authors. For instance, Ruggles considered the binomial expansions of $(1 + \alpha)^n$ and $(1 + \beta)^n$, where $\alpha$ and $\beta$ are golden ratio and its negative reciprocal, respectively. Vajda proved the Ruggles result by induction on $n$. Gauthier introduced the differential operator method to explore sums of the form $\sum_{i=1}^n i^m F_i$ for a fixed $m \in \mathbb{N}$ (although for $m \geq 4$, the closed forms become unwieldly with $m+2$ complicated summands). Benjamin and Quinn employed combinatorial reasoning techniques involving tiling a board of length $n$ with only squares and dominoes. To prove some of our main results (Theorems~\ref{thm:weighted_sum_U_seq_when_p_minus_q_is_1}, \ref{thm:weighted_sum_V_seq_when_p_minus_q_is_1}, \ref{thm:weighted_sum_U_seq_when_p_minus_q_is_not_equal_to_1}, and \ref{thm:weighted_sum_V_seq_when_p_minus_q_is_not_equal_to_1}), we employ a technique different than the latter methods; we utilize Abel's summation by parts technique.

When $\omega_i = 1$ for all $1 \leq i \leq n$, we simply have the sum of the first $n$ sequence terms $\sum_{i=1}^n S_i$, and this has been well-studied for a large number of popular sequences. When $\GenericSeq$ is either the Fibonacci or Lucas sequence and we let $\omega_i = i$, closed forms are well known (see for example Koshy~\cite[Identities~(25.1) and (25.2)]{Koshy2018}). However, a much earlier closed form was given by C.~R.~Wall in 1964 in the \textit{Fibonacci Quarterly} for the generalized Fibonacci sequence $\GibSeq$, defined by the recurrence $G_n = G_{n-1} + G_{n-2}$ with arbitrary integral initial conditions $G_0$ and $G_1$. His result was the following~\cite[Problem~B-40]{CR-B-40-Wall1964}:
$$ \sum_{i=1}^n i G_i = n G_{n+2} - G_{n+3} + G_3.$$
The goal of our paper is to generalize the latter weighted sum to the setting of the Lucas sequences of the first and second kind, namely $\LucasFormalUSeq$ and $\LucasFormalVSeq$, respectively. To accomplish this, we first give closed forms for the sum of consecutive Lucas sequences $\sum_{i=1}^n U_i(p,q)$ and $\sum_{i=1}^n V_i(p,q)$ and then use Abel's summation by parts method to derive closed forms for the weighted sums $\sum_{i=1}^n i U_i(p,q)$ and $\sum_{i=1}^n i V_i(p,q)$. Set $U_i \coloneq U_i(p,q)$ and $V_i \coloneq V_i(p,q)$ for brevity. The paper is organized as follows. In Section~\ref{sec:preliminaries}, we provide the definitions of the sequences $\LucasUSeq$ and $\LucasVSeq$, relevant sequence identities, establish our convention regarding nondegenerate Lucas sequences, and give 10 well-known Lucas sequences to which we apply our main results in Section~\ref{sec:applications_to_well_known_sequences}. In Section~\ref{sec:main_results_consecutive_sums}, we derive the following closed forms for the sums of consecutive terms:
\begin{align*}
\sum_{i=1}^n U_i &= \begin{dcases}\frac{q^{n+1} - (q-1) n - q}{(q-1)^2}  = \frac{q U_n - n}{q-1},  &\text{if $p - q = 1$;} \;\; \text{(Theorem~\ref{thm:consec_sum_U_seq_when_p_minus_q_is_1})}\\
\frac{U_{n+1} - q U_n - 1}{p-q-1}, &\text{if $p - q \neq 1$;} \;\; \text{(Theorem~\ref{thm:consec_sum_U_seq_when_p_minus_q_is_not_equal_to_1})}
\end{dcases}\\
\\
\sum_{i=1}^n V_i &= \begin{dcases} U_{n+1} + n - 1,  &\text{if $p - q = 1$;} \;\; \text{(Theorem~\ref{thm:consec_sum_V_seq_when_p_minus_q_is_1})}\\
\frac{V_{n+1} - q V_n - p + 2q}{p-q-1}, &\text{if $p - q \neq 1$}. \;\; \text{(Theorem~\ref{thm:consec_sum_V_seq_when_p_minus_q_is_not_equal_to_1})}
\end{dcases}
\end{align*}
In Section~\ref{sec:main_results_weighted_sums}, we derive the following closed forms for the weighted sums of consecutive terms: 
\begin{align*}
\sum_{i=1}^n i U_i &= \begin{dcases}\frac{q}{q-1} \biggl( n U_n - \frac{q U_{n-1} - n + 1}{q-1} - \frac{1}{q} \cdot \binom{n+1}{2} \biggr),  &\text{if $p - q = 1$;} \;\; \text{(Theorem~\ref{thm:weighted_sum_U_seq_when_p_minus_q_is_1})}\footnotemark \\
\frac{1}{p-q-1} \Bigl( n(U_{n+1} - q U_n) - \frac{\Omega(n)}{p-q-1} \Bigr), &\text{if $p - q \neq 1$;} \;\; \text{(Theorem~\ref{thm:weighted_sum_U_seq_when_p_minus_q_is_not_equal_to_1})}
\end{dcases}\\
&\;\;\;\;\;\;\text{ where } \Omega(n) = U_{n+1} - 2q U_n + q^2 U_{n-1} + q - 1,\\
\\
\sum_{i=1}^n i V_i &= \begin{dcases} n U_{n+1} - \frac{q U_n - n}{q-1} + \binom{n+1}{2},  &\text{if $p - q = 1$;} \;\; \text{(Theorem~\ref{thm:weighted_sum_V_seq_when_p_minus_q_is_1})}\footnotemark\\
\frac{1}{p-q-1} \biggl( n (V_{n+1} - q V_n) + 2 q - \frac{\Psi(n)}{p-q-1} \biggr), &\text{if $p - q \neq 1$;} \;\; \text{(Theorem~\ref{thm:weighted_sum_V_seq_when_p_minus_q_is_not_equal_to_1})}
\end{dcases}\\
&\;\;\;\;\;\;\text{ where } \Psi(n) = V_{n+1} - 2q V_n + q^2 V_{n-1} +(p-2q)(q-1).
\end{align*}
\vspace{-.28in}
\footnotetext[5]{We also derive an equivalent closed form $\sum_{i=1}^n i U_i = \frac{q}{q-1} \left( \frac{q^n (nq-n-1) + 1}{(q-1)^2} - \frac{1}{q} \binom{n+1}{2} \right)$ in Corollary~\ref{cor:weighted_sum_U_seq_when_p_minus_q_is_1}.}
\footnotetext[6]{We also derive an equivalent closed form $\sum_{i=1}^n i V_i = \frac{q^{n+1} (nq-n-1)+q}{(q-1)^2} + \binom{n+1}{2}$ in Corollary~\ref{cor:weighted_sum_V_seq_when_p_minus_q_is_1}.}

\noindent Then in Section~\ref{sec:applications_to_well_known_sequences}, we apply our weighted sum formulas to 10 well-known Lucas sequences: Fibonacci $\FibSeq$, Lucas $\LucSeq$, Pell $\PellSeq$, companion Pell $\QellSeq$, balancing $\BalSeq$, double Lucas-balancing $\DoubleLucBalSeq$, Jacobstahl $\JacSeq$, Jacobstahl-Lucas $\JacLucSeq$, Mersenne $\MerSeq$, and Mersenne-Lucas $\MerLucSeq$. We not only recover known results for $\sum_{i=1}^n i F_i$ and $\sum_{i=1}^n i L_i$ in Theorem~\ref{thm:Fib_Luc_weighted_sum}, but also introduce new results for the remaining eight sequences in Theorems~\ref{thm:Pell_Qell_weighted_sum}, \ref{thm:Bal_DoubleLucBal_weighted_sum}, \ref{thm:Jac_JacLuc_weighted_sum}, and \ref{thm:Mer_MerLuc_weighted_sum}, as follows:
\begin{align*}
&\sum_{i=1}^n i F_i = n F_{n+2} - F_{n+3} + F_3 &&\sum_{i=1}^n i L_i = n L_{n+2} - L_{n+3} + L_3\\
&\sum_{i=1}^n i P_i = \frac{1}{2} \Bigl( (n-1) P_{n+1} + n P_{n} + 1\Bigr) &&\sum_{i=1}^n i Q_i = \frac{1}{2} \Bigl( (n-1) Q_{n+1} + n Q_{n} + 2 \Bigr)\\
&\sum_{i=1}^n i B_i = \frac{1}{4} \Bigl( n B_{n+1} - (n+1) B_{n} \Bigr) &&\sum_{i=1}^n i \DLB_i = \frac{1}{4} \Bigl( n \DLB_{n+1} - (n+1) \DLB_{n} + 2 \Bigr)\\
&\sum_{i=1}^n i J_i = \frac{1}{2} \Bigl( n J_{n+2} - \frac{J_{n+3}-3}{2} \Bigr) &&\sum_{i=1}^n i j_i = \frac{1}{2} \Bigl( n j_{n+2} - 4 - \frac{j_{n+3}-15}{2} \Bigr)\\
&\sum_{i=1}^n i \Mer_i = 2^{n+1}(n-1) + 2 - \binom{n+1}{2} &&\sum_{i=1}^n i \MerLuc_i = 2^{n+1}(n-1) + 2 + \binom{n+1}{2}.
\end{align*}
Finally in Section~\ref{sec:open_questions}, we give further directions that the author is currently exploring and offer some tantalizing open questions for the motivated reader.

%%%%%%%%%%%%%%%%%%%%%%%%%%%%%%%%%%%%%
%%%%%%%%%%%%%%%%%%%%%%%%%%%%%%%%%%%%%
%%%%%%%%%    SECTION 2    %%%%%%%%%%%
%%%%%%%%%%%%%%%%%%%%%%%%%%%%%%%%%%%%%
%%%%%%%%%%%%%%%%%%%%%%%%%%%%%%%%%%%%%

%\newpage

\section{Preliminaries}\label{sec:preliminaries}

\begin{definition}[Lucas sequences]
Let $p$ and $q$ be nonzero (not necessarily coprime) integers. The \textit{Lucas sequences of the first and second kind}, respectively, are denoted $\LucasFormalUSeq$ and $\LucasFormalVSeq$. Setting $U_n \coloneq U_n(p,q)$ and $V_n \coloneq V_n(p,q)$, the sequences $\LucasUSeq$ and $\LucasVSeq$ satisfy the recurrence relations $U_n = p U_{n-1} - q U_{n-2}$ and $V_n = p V_{n-1} - q V_{n-2}$, respectively, with initial conditions $U_0 = 0$, $U_1 = 1$, $V_0 = 2$, and $V_1=p$.
\end{definition}

\begin{remark}
Ballot and Williams assert that although Lucas and subsequently many other authors introduce the constraint that $\gcd(p,q)=1$, the mathematical literature contains many papers that prove results about the Lucas sequences with $\gcd(p,q)=1$, which hold just as well for Lucas sequences where $\gcd(p,q) \neq 1$~\cite[p.~13]{Ballot-Williams2023}.
\end{remark}

The well-known Binet forms for the Lucas sequences, as well as some identities on the roots of the sequences' characteristic polynomial, are given in the following proposition.

\begin{proposition}[Lucas~{\cite[Identity~(2)]{Lucas1878}}]\label{prop:identities_on_the_roots_of_the_characteristic_equation}
Let $\alpha$ and $\beta$ be roots of the \textit{characteristic polynomial} $x^2 - px + q$ in the quadratic field $\mathbb{Q}(\sqrt{\Delta})$ where $\Delta = p^2 - 4q$ is the \textit{discriminant} of the polynomial. Then $\alpha = \frac{p + \sqrt{\Delta}}{2}$ and $\beta = \frac{p - \sqrt{\Delta}}{2}$. Moreover, the Lucas sequences $\LucasFormalUSeq$ and $\LucasFormalVSeq$, respectively, are given by
$$ U_n(p,q) = \frac{\alpha^n - \beta^n}{\alpha - \beta} \;\;\;\text{and}\;\;\; V_n(p,q) = \alpha^n + \beta^n.$$
\end{proposition}

\boitegrise{
\begin{convention}\label{conv:nondegenerate_Lucas_sequences}
\vspace{-.15in}
In this paper, we assume that  $\LucasFormalUSeq$ and $\LucasFormalVSeq$ are \textit{nondegenerate}. That is, $q \neq 0$ and the ratio $\frac{\alpha}{\beta}$ is not a root of unity. In particular, this implies that $\alpha$ and $\beta$ are distinct
and hence $\Delta \neq 0$. For $\gcd(p,q)=1$, a Lucas sequence is degenerate if $(p,q) \in \{(\pm 2,1)$, $(\pm 1,1)$, $(0,\pm 1)$, and $(\pm 1,0)\}$~\cite[pp.~5--6]{Ribenboim2000}.
\vspace{-.3in}
\end{convention}}{0.95\textwidth}

In Table~\ref{table:10_well_known_Lucas_sequences}, we give the first 10 terms and their corresponding sequence entries in the OEIS~\cite{Sloane-OEIS} of the 10 Lucas sequences we explore in this paper in the following list:
\begin{align*}
\left(U_n(1,-1)\right)_{n \geq 0} &= \FibSeq &&\hspace{-.5in}\text{Fibonacci}\\
\left(V_n(1,-1)\right)_{n \geq 0} &= \LucSeq &&\hspace{-.5in}\text{Lucas}\\
\left(U_n(2,-1)\right)_{n \geq 0} &= \PellSeq &&\hspace{-.5in}\text{Pell}\\
\left(V_n(2,-1)\right)_{n \geq 0} &= \QellSeq &&\hspace{-.5in}\text{companion Pell} &&\hspace{-.5in}\text{(see Remark~\ref{rem:Pell_Lucas_confusion})}\\
\left(U_n(6,1)\right)_{n \geq 0} &= \BalSeq &&\hspace{-.5in}\text{balancing}\\
\left(V_n(6,1)\right)_{n \geq 0} &= \DoubleLucBalSeq &&\hspace{-.5in}\text{double Lucas-balancing} &&\hspace{-.5in}\text{(see Remark~\ref{rem:Lucas_balancing_confusion})}\\
\left(U_n(1,-2)\right)_{n \geq 0} &= \JacSeq &&\hspace{-.5in}\text{Jacobstahl}\\
\left(V_n(1,-2)\right)_{n \geq 0} &= \JacLucSeq &&\hspace{-.5in}\text{Jacobstahl-Lucas}\\
\left(U_n(3,2)\right)_{n \geq 0} &= \MerSeq &&\hspace{-.5in}\text{Mersenne}\\
\left(V_n(3,2)\right)_{n \geq 0} &= \MerLucSeq &&\hspace{-.5in}\text{Mersenne-Lucas.}
\end{align*}
We denote a Mersenne-Lucas number with the symbol $\MerLuc_n$ as it turns out that $\MerLuc_n = 2^n + 1$, which are the so-called ``Fermat numbers'' whenever $n = 2^k$ for some $k$.

%\vspace{-.25in}

\begin{table}[h!]
\renewcommand{\arraystretch}{1.1}
\centering
\begin{tabular}{|c||c|c|c|c|c|c|c|c|c|c||c|}
\hline
\blue{$n$} & \redbf{0} & \redbf{1} & \redbf{2} & \redbf{3} & \redbf{4} & \redbf{5} & \redbf{6} & \redbf{7} & \redbf{8} & \redbf{9} & \blue{OEIS}\\ \hline\hline
\rowcolor{lightgray}
\blue{$F_n$} & 0 & 1 & 1 & 2 & 3 & 5 & 8 & 13 & 21 & 34 & \seqnum{A000045}\\ \hline
\blue{$L_n$} & 2 & 1 & 3 & 4 & 7 & 11 & 18 & 29 & 47 & 76 & \seqnum{A000032}\\ \hline\hline
\rowcolor{lightgray}
\blue{$P_n$} & 0 & 1 & 2 & 5 & 12 & 29 & 70 & 169 & 408 & 985 & \seqnum{A000129}\\ \hline
\blue{$Q_n$} & 2 & 2 & 6 & 14 & 34 & 82 & 198 & 478 & 1154 & 2786 & \seqnum{A002203}\\ \hline\hline
\rowcolor{lightgray}
\blue{$B_n$} & 0 & 1 & 6 & 35 & 204 & 1189 & 6930 & 40391 & 235416 & 1372105 & \seqnum{A001109}\\ \hline
\blue{$\DLB_n$} & 2 & 6 & 34 & 198 & 1154 & 6726 & 39202 & 228486 & 1331714 & 7761798 & \seqnum{A001541} \\ \hline
\rowcolor{lightgray}
\blue{$J_n$} & 0 & 1 & 1 & 3 & 5 & 11 & 21 & 43 & 85 & 171 & \seqnum{A001045}\\ \hline
\blue{$j_n$} & 2 & 1 & 5 & 7 & 17 & 31 & 65 & 127 & 257 & 511 & \seqnum{A014551}\\ \hline\hline
\rowcolor{lightgray}
\blue{$\Mer_n$} & 0 & 1 & 3 & 7 & 15 & 31 & 63 & 127 & 255 & 511 & \seqnum{A000225}\\ \hline
\blue{$\MerLuc_n$} & 2 & 3 & 5 & 9 & 17 & 33 & 65 & 129 & 257 & 513 & \seqnum{A000051}\\ \hline
\end{tabular}
\caption{The first 10 Fibonacci $F_n$, Lucas $L_n$, Pell $P_n$, companion Pell $Q_n$, balancing $B_n$, double Lucas-balancing $\DLB_n$, Jacobstahl $J_n$, Jacobstahl-Lucas $j_n$, Mersenne $\Mer_n$, and Mersenne-Lucas $\MerLuc_n$ numbers}
\label{table:10_well_known_Lucas_sequences}
\end{table}

\begin{remark}\label{rem:Pell_Lucas_confusion}
The sequence $\QellSeq$ in Table~\ref{table:10_well_known_Lucas_sequences} is sometimes called the Pell-Lucas sequence. However, in the literature, there is some discrepancy on the precise definition of what is called the ``Pell-Lucas sequence''. Many sources attribute $\QellSeq$ to OEIS sequence \seqnum{A002203} where it is called the ``companion Pell sequence''. In 1976, Hoggatt and Alexanderson have one of the earlierst papers defining $\QellSeq$ as the Pell-Lucas sequence~\cite[Section~6]{Hoggatt1976}. On the other hand, Koshy~\cite{Koshy2014} and many other researchers  define the Pell-Lucas sequence with initial values $Q_0 = 1$ and $Q_1 = 1$, producing a sequence that is technically $\left(Q_n / 2\right)_{n \geq 0}$, sometimes called the ``half-companion Pell sequence'' or the ``associated Pell sequence'', which is not a Lucas sequence of the first or second kind. The latter is the OEIS sequence \seqnum{A001333}.
\end{remark}

\begin{remark}\label{rem:Lucas_balancing_confusion}
The well-studied Lucas-balancing sequence $\LucBalSeq$ was introduced by Behera and Panda in 1999 in the \textit{Fibonacci Quarterly} although not yet denoted $C_n$~\cite{Behera-Panda1999}. The Lucas-balancing number $C_n$ is related to the balancing number $B_n$ by the equality $C_n = \sqrt{8 B_n^2 + 1}$. However, the sequence $\LucBalSeq$ is not actually a Lucas sequence of the first or second kind, as its initial conditions are $C_0 = 1$ and $C_1 = 3$ with recurrence $C_n = 6 C_{n-1} - C_{n-2}$. However, if we double the terms of the sequence, we do get the Lucas sequence $\left(V_n(6,1)\right)_{n \geq 0}$, hence we call $\DoubleLucBalSeq$ the ``double Lucas-balancing sequence'' with $\DLB_n = 2C_n$, and $\DoubleLucBalSeq$ has recurrence $\DLB_n = 6 \DLB_{n-1} - \DLB_{n-2}$ with initial conditions $\DLB_0 = 2$ and $\DLB_1 = 6$.
\end{remark}

%%%%%%%%%%%%%%%%%%%%%%%%%%
%%%%%%%%%%%%%%%%%%%%%%%%%%
%%%%%%%%%%%%%%%%%%%%%%%%%%

%%%%%%%%%%%%%%%%%%%%%%%%%%%%%%%%%%%%%
%%%%%%%%%%%%%%%%%%%%%%%%%%%%%%%%%%%%%
%%%%%%%%%    SECTION 3    %%%%%%%%%%%
%%%%%%%%%%%%%%%%%%%%%%%%%%%%%%%%%%%%%
%%%%%%%%%%%%%%%%%%%%%%%%%%%%%%%%%%%%%

%\newpage

\section{Main results for the consecutive sums}\label{sec:main_results_consecutive_sums}

\boitegrise{
\begin{convention}\label{conv:notation_for_Lucas_sequences}
\vspace{-.15in}
For brevity in the remainder of the paper, we set $U_n\coloneq\LucasFormalUTerm$ and $V_n\coloneq\LucasFormalVTerm$ when it is clear what the $p$ and $q$ values are. Therefore, we often write the Lucas sequences $\LucasFormalUSeq$ and $\LucasFormalVSeq$, respectively, as $\LucasUSeq$ and $\LucasVSeq$.
\vspace{-.3in}
\end{convention}}{0.95\textwidth}

To provide our main results for the sums $\sum_{i=1}^n U_i$ and $\sum_{i=1}^n V_i$, we first handle the case when $p-q=1$ in Subsection~\ref{subsec:consecutive_sum_when_p_minus_q_equals_1}. Then we handle the case when $p-q \neq 1$ in Subsection~\ref{subsec:consecutive_sum_when_p_minus_q_does_not_equal_1}.

%%%%%%%%%%%%%%%%%%%%%%%%%%
%%%%%%%%%%%%%%%%%%%%%%%%%%
%%%%%%%%%%%%%%%%%%%%%%%%%%

\subsection{Sum of consecutive terms when \texorpdfstring{$p-q=1$}{p-q=1}}\label{subsec:consecutive_sum_when_p_minus_q_equals_1}

Recall that we are excluding the degenerate Lucas sequences as per Convention~\ref{conv:nondegenerate_Lucas_sequences}. But since we assume that $p-q=1$ in this subsection, this means we are excluding the cases when $(p,q)$ equals $(0,-1)$, $(1,0)$, or $(2,1)$. However, the following lemma holds even for the degenerate Lucas sequences.

\begin{lemma}\label{lem:U_and_V_seq_when_p_minus_q_is_1}
Let $p$ and $q$ be integers such that $p - q = 1$. Then for all $n \geq 0$, we have
\begin{align}
    U_n &=  \frac{q^n - 1}{q-1} = 1 + q + q^2 + \cdots + q^{n-1} \label{eq:U_seq_when_p_minus_q_is_1}\\
    V_n &= q^n + 1. \label{eq:V_seq_when_p_minus_q_is_1}
\end{align}
\end{lemma}

\begin{proof}
Since $p - q = 1$, we have $p = q + 1$ and thus $\Delta = p^2 - 4q = (q+1)^2 - 4q = (q-1)^2$. By Proposition~\ref{prop:identities_on_the_roots_of_the_characteristic_equation}, it follows that
$$ \alpha = \frac{p + \sqrt{\Delta}}{2} = \frac{(q+1) + (q-1)}{2} = q \;\;\text{ and }\;\; \beta = \frac{p - \sqrt{\Delta}}{2} = \frac{p-(q-1)}{2} = 1.$$
Hence by the Binet form for $U_n$ in Proposition~\ref{prop:identities_on_the_roots_of_the_characteristic_equation} and a finite geometric series argument,
$$U_n = \frac{\alpha^n - \beta^n}{\alpha - \beta} = \frac{q^n - 1}{q-1} = 1 + q + q^2 + \cdots + q^{n-1}, $$
proving Equation~\eqref{eq:U_seq_when_p_minus_q_is_1}. Furthermore, by the Binet form for $V_n$ in Proposition~\ref{prop:identities_on_the_roots_of_the_characteristic_equation}, we have $V_n = \alpha^n + \beta^n = q^n +1$, proving Equation~\eqref{eq:V_seq_when_p_minus_q_is_1}.
\end{proof}

In the next theorem, the closed form is an indeterminate whenever $q=1$, so we eliminate that choice. This is fine since $(p,q) = (2,1)$ yields a degenerate sequence $\LucasFormalUSeq$.

\begin{theorem}\label{thm:consec_sum_U_seq_when_p_minus_q_is_1}
Let $p$ and $q$ be integers such that $q \neq 1$ and $p - q = 1$, and hence $\Delta = (q-1)^2$ is the discriminant of the characteristic polynomial $x^2 - px + q$. Then for all $n \geq 1$, we have
\begin{align}
\sum_{i=1}^n U_i  = \frac{q^{n+1} - (q-1) n - q}{\Delta}  = \frac{q U_n - n}{q-1}. \label{eq:sum_of_n_consec_U_seq_terms}
\end{align}
\end{theorem}

\begin{proof}
Observe the following sequence of equalities:
\begin{align*}
\sum_{i=1}^n U_i &= \frac{1}{q-1} \cdot \sum_{i=1}^n (q^i - 1) &\text{by Equation~\eqref{eq:U_seq_when_p_minus_q_is_1} of Lemma~\ref{lem:U_and_V_seq_when_p_minus_q_is_1}}\\
    &= \frac{1}{q-1} \cdot \left( \sum_{i=1}^n q^i - \sum_{i=1}^n 1 \right)\\
    &= \frac{1}{q-1} \cdot \left( \left( \frac{q^{n+1} - 1}{q - 1} - 1 \right) - n \right)\\
    &= \frac{q^{n+1} - 1}{(q-1)^2} - \frac{n+1}{q-1}\\
    &= \frac{q^{n+1} - 1 - (n+1)(q-1)}{(q-1)^2}\\
    &= \frac{q^{n+1} - (q-1) n - q}{\Delta},
\end{align*}
where the third equality holds by a finite geometric series argument, and the last equality holds since $p - q = 1$ implies $\Delta = (q-1)^2$, proving the first equality in Equation~\eqref{eq:sum_of_n_consec_U_seq_terms}. Lastly, observe that
\begin{align*}
\frac{q^{n+1} - (q-1) n - q}{\Delta}
    &= \frac{q(q^n - 1) - (q-1) n}{(q-1)^2}\\
    &= \frac{q(q - 1)(1 + q + \cdots + q^{n-1}) - (q-1) n}{(q-1)^2}\\
    &= \frac{q(1 + q + \cdots + q^{n-1}) - n}{q-1}\\
    &= \frac{q U_n - n}{q-1},
\end{align*}
where the second equality holds by a finite geometric series argument, and the last equality holds by Equation~\eqref{eq:U_seq_when_p_minus_q_is_1} of Lemma~\ref{lem:U_and_V_seq_when_p_minus_q_is_1}, proving the second equality in Equation~\eqref{eq:sum_of_n_consec_U_seq_terms}.
\end{proof}

\begin{theorem}\label{thm:consec_sum_V_seq_when_p_minus_q_is_1}
Let $p$ and $q$ be integers such that $p - q = 1$. Then for all $n \geq 1$, we have
$$ \sum_{i=1}^n V_i  = U_{n+1} + n - 1 .$$
\end{theorem}

\begin{proof}
Observe the following sequence of equalities:
\begin{align*}
\sum_{i=1}^n V_i = \sum_{i=1}^n (q^i + 1) = \left( \sum_{i=0}^n q^i \right) - 1 + \sum_{i=1}^n 1 = \frac{q^{n+1} - 1}{q-1} + n - 1 = U_{n+1} + n - 1,
\end{align*}
where the first equality holds by Equation~\eqref{eq:V_seq_when_p_minus_q_is_1} of Lemma~\ref{lem:U_and_V_seq_when_p_minus_q_is_1}, the third equality holds by a finite geometric series argument, and the last equality holds by Equation~\eqref{eq:U_seq_when_p_minus_q_is_1} of Lemma~\ref{lem:U_and_V_seq_when_p_minus_q_is_1}. Thus the result follows.
\end{proof}

%%%%%%%%%%%%%%%%%%%%%%%%%%
%%%%%%%%%%%%%%%%%%%%%%%%%%
%%%%%%%%%%%%%%%%%%%%%%%%%%

\subsection{Sum of consecutive terms when \texorpdfstring{$p-q \neq 1$}{p-q is not equal to 1}}\label{subsec:consecutive_sum_when_p_minus_q_does_not_equal_1}

Closed forms for $\sum_{i=1}^n U_i$ and $\sum_{i=1}^n V_i$ were not explicitly given in Lucas' seminal 1878 paper~\cite{Lucas1878}; however, they arise as a consequence of a sequence of Lucas' deep identities culminating in his Identity~(56), which was stated incorrectly in his paper (see Appendix~\ref{app:correct_proof_for_Lucas_Identity_(56)}), but stated correctly in the following two equations:
\begin{align}
    \sum_{i=1}^n U_{ir} &= \frac{U_r + q^r U_{nr} - U_{(n+1)r}}{1 + q^r - V_r} \label{eq:Lucas_corrected_U_identity}\\
    \sum_{i=1}^n V_{ir} &= \frac{V_r + q^r V_{nr} - V_{(n+1)r} - 2q^r}{1 + q^r - V_r}. \label{eq:Lucas_corrected_V_identity}
\end{align}
Setting $r\coloneq 1$ and multiplying the numerators and denominators by $-1$ in Equations~\eqref{eq:Lucas_corrected_U_identity} and \eqref{eq:Lucas_corrected_V_identity}, respectively, Theorems~\ref{thm:consec_sum_U_seq_when_p_minus_q_is_not_equal_to_1} and \ref{thm:consec_sum_V_seq_when_p_minus_q_is_not_equal_to_1} follow. However, we give simple induction proofs for both theorems that rely solely on the recurrence relation for $\LucasUSeq$ and $\LucasVSeq$ and no identities in Lucas' paper.

\begin{theorem}\label{thm:consec_sum_U_seq_when_p_minus_q_is_not_equal_to_1}
Let $p$ and $q$ be integers such that $p - q \neq 1$. Then for all $n \geq 1$, we have
$$ \sum_{i=1}^n U_i  = \frac{U_{n+1} - q U_n - 1}{p-q-1} .$$
\end{theorem}

\begin{proof}
We induct on $n$. For $n =1$, it follows that
\begin{align*}
\sum_{i=1}^1 U_i = U_1 = \frac{p U_1 - q U_1 - 1}{p - q - 1} = \frac{U_2 - q U_1 -1}{p - q - 1},
\end{align*}
where the second equality holds since $U_1 = 1$, and the third equality holds since $U_0 = 0$ along with the recurrence $U_2 = p U_1 - q U_0$ imply that $U_2 = p U_1$. Now assume that for some $k \geq 1$, we have $\sum_{i=1}^k U_i  = \frac{U_{k+1} - q U_k - 1}{p-q-1}$, and observe the following sequence of equalities:
\begin{align*}
\sum_{i=1}^{k+1} U_i &= \left( \sum_{i=1}^k U_i \right ) + U_{k+1}\\
    &= \frac{U_{k+1} - q U_k - 1}{p-q-1} + U_{k+1}\\
    &= \frac{U_{k+1} - q U_k - 1 + (p - q - 1)U_{k+1}}{p-q-1}\\
    &= \frac{U_{k+1} - q U_k - 1 + p U_{k+1} - q U_{k+1} - U_{k+1}}{p-q-1}\\
    &= \frac{ (p U_{k+1} - q U_k) - q U_{k+1} - 1}{p-q-1}\\
    &= \frac{ U_{k+2} - q U_{k+1} - 1}{p-q-1},
\end{align*}
where the second equality holds by the induction hypothesis, and the last equality holds by the recurrence for $\LucasUSeq$. Thus the result follows.
\end{proof}

\begin{theorem}\label{thm:consec_sum_V_seq_when_p_minus_q_is_not_equal_to_1}
Let $p$ and $q$ be integers such that $p - q \neq 1$. Then for all $n \geq 1$, we have
$$ \sum_{i=1}^n V_i  = \frac{V_{n+1} - q V_n - p + 2q}{p-q-1} .$$
\end{theorem}

\begin{proof}
We induct on $n$. For $n =1$, it follows that
\begin{align*}
\sum_{i=1}^1 V_i = V_1 = p \cdot \frac{p - q - 1}{p - q - 1} = \frac{p^2 - pq - p}{p - q - 1} = \frac{(p^2 - 2q) - pq - p + 2q}{p - q - 1} = \frac{V_2 - q V_1 - p + 2q}{p - q - 1},
\end{align*}
where the last equality holds since $V_1 = p$ and $V_2 = p^2 - 2q$. Now assume that for some $k \geq 1$, we have $\sum_{i=1}^k V_i  = \frac{V_{k+1} - q V_k - p + 2q}{p-q-1}$, and observe the following sequence of equalities:
\begin{align*}
\sum_{i=1}^{k+1} V_i &= \left( \sum_{i=1}^k V_i \right ) + V_{k+1}\\
    &= \frac{V_{k+1} - q V_k - p + 2q}{p-q-1} + V_{k+1}\\
    &= \frac{V_{k+1} - q V_k - p + 2q + (p - q - 1) V_{k+1}}{p - q - 1}\\
    &= \frac{V_{k+1} - q V_k - p + 2q + p V_{k+1} - q V_{k+1} - V_{k+1}}{p - q - 1}\\
    &= \frac{(p V_{k+1} - q V_k) - q V_{k+1} - p + 2q}{p - q - 1}\\
    &= \frac{V_{k+2} - q V_{k+1} - p + 2q}{p - q - 1},
\end{align*}
where the second equality holds by the induction hypothesis, and the last equality holds by the recurrence for $\LucasVSeq$. Thus the result follows.
\end{proof}

%%%%%%%%%%%%%%%%%%%%%%%%%%%%%%%%%%%%%
%%%%%%%%%%%%%%%%%%%%%%%%%%%%%%%%%%%%%
%%%%%%%%%    SECTION 4    %%%%%%%%%%%
%%%%%%%%%%%%%%%%%%%%%%%%%%%%%%%%%%%%%
%%%%%%%%%%%%%%%%%%%%%%%%%%%%%%%%%%%%%

%\newpage

\section{Main results for the weighted sums}\label{sec:main_results_weighted_sums}

A main tool we use in this section is Abel's summation by parts formula, which transforms a finite sum $\sum_{i=1}^n a_i b_i$ of products of two sequences $(a_i)_{i \geq 1}$ and $(b_i)_{i \geq 1}$ into a sum of  simpler summations. In our case, these simpler summations are the consecutive sums we derived already in Section~\ref{sec:main_results_consecutive_sums}. As there are a variety of forms in which Abel's formula can be presented, we state the version of the formula that we employ.

\begin{lemma}[Abel]\label{lem:Abel_summation}
Let $(a_i)_{i \geq 1}$ and $(b_i)_{i \geq 1}$ be sequences of real numbers. It follows that
$$ \sum_{i=1}^n a_i b_i = a_n \cdot \sum_{i=1}^n b_i + \sum_{i=1}^{n-1} \Bigl( (a_i - a_{i+1}) \cdot \sum_{j=1}^i b_j \Bigr).$$
\end{lemma}

To simplify the closed forms, we often employ the following pleasing combinatorial identity, bearing a myriad of straight-forward proofs, which we leave to the motivated reader.

\begin{lemma}\label{lem:cute_combinatorial_identity}
Let $n \geq 1$ be given. Then the following identity holds:
$$ n^2 = \binom{n}{2} + \binom{n+1}{2}. $$
\end{lemma}

%%%%%%%%%%%%%%%%%%%%%%%%%%
%%%%%%%%%%%%%%%%%%%%%%%%%%
%%%%%%%%%%%%%%%%%%%%%%%%%%

\subsection{Weighted sum of consecutive terms when \texorpdfstring{$p-q=1$}{p-q=1}}\label{subsec:weighted_sum_when_p_minus_q_equals_1}

\begin{theorem}\label{thm:weighted_sum_U_seq_when_p_minus_q_is_1}
Let $p$ and $q$ be integers such that $q \neq 1$ and $p - q = 1$. Then for all $n \geq 1$, we have
\begin{align}
\sum_{i=1}^n i U_i = \frac{q}{q-1} \biggl( n U_n - \frac{q U_{n-1} - n + 1}{q-1} - \frac{1}{q} \cdot \binom{n+1}{2} \biggr).
\end{align}
\end{theorem}

\begin{proof}
Setting $a_i \coloneq i$ and $b_i \coloneq U_i$ for all $1 \leq i \leq n$ in Lemma~\ref{lem:Abel_summation}, we get the following:
\begin{align*}
\sum_{i=1}^n i U_i &= n \cdot \sum_{i=1}^n U_i + \sum_{i=1}^{n-1} \Bigl( (i - (i+1)) \cdot \sum_{j=1}^i U_j \Bigr)\\
    &= n \cdot \frac{q U_n - n}{q-1} + \sum_{i=1}^{n-1} (-1) \cdot \frac{q U_i - i}{q-1} &\text{by Theorem~\ref{thm:consec_sum_U_seq_when_p_minus_q_is_1}}\\
    &= \frac{1}{q-1} \biggl( n q U_n - n^2 - q \sum_{i=1}^{n-1} U_i + \sum_{i=1}^{n-1} i \biggr)\\
    &= \frac{1}{q-1} \biggl( n q U_n - n^2 - q \cdot \frac{q U_{n-1} - (n - 1)}{q-1} +  \sum_{i=1}^{n-1} i \biggr) &\text{by Theorem~\ref{thm:consec_sum_U_seq_when_p_minus_q_is_1}}\\
    &= \frac{1}{q-1} \biggl( n q U_n - q \cdot \frac{q U_{n-1} - n + 1}{q-1} +  \binom{n}{2} - n^2 \biggr)\\
    &= \frac{1}{q-1} \biggl( n q U_n - q \cdot \frac{q U_{n-1} - n + 1}{q-1} - \binom{n+1}{2} \biggr)\\
    &= \frac{q}{q-1} \biggl( n U_n - \frac{q U_{n-1} - n + 1}{q-1} - \frac{1}{q} \cdot \binom{n+1}{2} \biggr),
\end{align*}
where the second to last equality holds by Lemma~\ref{lem:cute_combinatorial_identity}, and the last equality holds since $(p,q) = (1,0)$ yields a degenerate sequence $\LucasFormalUSeq$ (recall Convention~\ref{conv:nondegenerate_Lucas_sequences}) and hence $q$ is nonzero. Thus the result follows.
\end{proof}

By utilizing the Binet form for $U_n$ (when $p-q=1$) given in Lemma~\ref{lem:U_and_V_seq_when_p_minus_q_is_1} and some algebraic simplifications, we can rewrite the closed form for $\sum_{i=1}^n i U_i$ alternately as follows.

\begin{corollary}\label{cor:weighted_sum_U_seq_when_p_minus_q_is_1}
Let $p$ and $q$ be integers such that $p - q = 1$. Then for all $n \geq 1$, we have
$$ \sum_{i=1}^n i U_i  =  \frac{q}{q-1} \biggl( \frac{q^n (nq-n-1) + 1}{\Delta} - \frac{1}{q} \binom{n+1}{2} \biggr), $$
where $\Delta = (q-1)^2$ is the discriminant of the characteristic polynomial $x^2 - px + q$.
\end{corollary}

\begin{proof}
It suffices to show that $n U_n - \frac{q U_{n-1} - n + 1}{q-1} = \frac{q^n (nq-n-1) + 1}{\Delta}$. By Lemma~\ref{lem:U_and_V_seq_when_p_minus_q_is_1}, we have
\begin{align*}
n U_n - \frac{q U_{n-1} - n + 1}{q-1} &= n \cdot \frac{q^n - 1}{q-1} - \frac{q \left( \frac{q^{n-1} - 1}{q-1} \right) - (n - 1) \cdot \frac{q-1}{q-1}}{q-1}\\
    &= \frac{n (q^n - 1) (q-1)}{(q-1)^2} - \frac{q^n - nq + n - 1}{(q-1)^2}\\
    &= \frac{n (q^{n+1} - q^n - q + 1) - (q^n - nq + n - 1)}{\Delta}\\
    &= \frac{n q^{n+1} - n q^n - q^n + 1}{\Delta}\\
    &= \frac{q^n (nq-n-1) + 1}{\Delta},
\end{align*}
and thus the result follows.
\end{proof}

\begin{theorem}\label{thm:weighted_sum_V_seq_when_p_minus_q_is_1}
Let $p$ and $q$ be integers such that $p - q = 1$. Then for all $n \geq 1$, we have
$$ \sum_{i=1}^n i V_i  = n U_{n+1} - \frac{q U_n - n}{q-1} + \binom{n+1}{2} .$$
\end{theorem}

\begin{proof}
Setting $a_i \coloneq i$ and $b_i \coloneq U_i$ for all $1 \leq i \leq n$ in Lemma~\ref{lem:Abel_summation}, we get the following:
\begin{align*}
\sum_{i=1}^n i V_i &= n \cdot \sum_{i=1}^n V_i + \sum_{i=1}^{n-1} \Bigl( (i - (i+1)) \cdot \sum_{j=1}^i V_j \Bigr)\\
    &= n (U_{n+1} + n - 1) + \sum_{i=1}^{n-1} (-1) \cdot (U_{i+1} + i - 1) &\text{by Theorem~\ref{thm:consec_sum_V_seq_when_p_minus_q_is_1}}\\
    &= n U_{n+1} + n^2 - n - \sum_{i=1}^{n-1} U_{i+1} - \sum_{i=1}^{n-1} i + \sum_{i=1}^{n-1} 1\\
    &= n U_{n+1} + n^2 - n - \sum_{i=1}^n U_i + U_1 - \binom{n}{2} + n - 1\\
    &= n U_{n+1} + n^2 - \frac{q U_n - n}{q-1} - \binom{n}{2} &\text{by Theorem~\ref{thm:consec_sum_U_seq_when_p_minus_q_is_1}}\\
    &= n U_{n+1} - \frac{q U_n - n}{q-1} + \binom{n+1}{2},
\end{align*}
where the second to last equality holds since $U_1 = 1$, and the last equality holds by Lemma~\ref{lem:cute_combinatorial_identity}. Thus the result follows.
\end{proof}

By utilizing the Binet form for $U_n$ (when $p-q=1$) given in Lemma~\ref{lem:U_and_V_seq_when_p_minus_q_is_1} and some algebraic simplifications, we can rewrite the closed form for $\sum_{i=1}^n i V_i$ alternately as follows.

\begin{corollary}\label{cor:weighted_sum_V_seq_when_p_minus_q_is_1}
Let $p$ and $q$ be integers such that $p - q = 1$. Then for all $n \geq 1$, we have
$$ \sum_{i=1}^n i V_i  = \frac{q^{n+1} (nq-n-1)+q}{\Delta} + \binom{n+1}{2}, $$
where $\Delta = (q-1)^2$ is the discriminant of the characteristic polynomial $x^2 - px + q$.
\end{corollary}

\begin{proof}
It suffices to show that $n U_{n+1} - \frac{q U_n - n}{q-1} = \frac{q^{n+1} (nq-n-1)+q}{\Delta}$. By Lemma~\ref{lem:U_and_V_seq_when_p_minus_q_is_1}, we have
\begin{align*}
n U_{n+1} - \frac{q U_n - n}{q-1} &= n \cdot \frac{q^{n+1} - 1}{q-1} - \frac{q \left( \frac{q^n - 1}{q-1} \right) - n \cdot \frac{q-1}{q-1}}{q-1}\\
    &= \frac{n (q^{n+1} - 1) (q-1)}{(q-1)^2} - \frac{q^{n+1} - q - nq + n}{(q-1)^2}\\
    &= \frac{1}{\Delta} \Bigl( n (q^{n+2} - q^{n+1} - q + 1 ) - q^{n+1} + q + nq - n \Bigr)\\
    &= \frac{1}{\Delta} \Bigl( n q^{n+2} - (n+1) q^{n+1} + q \Bigr)\\
    &= \frac{q^{n+1} (nq - n - 1) + q}{\Delta},
\end{align*}
and thus the result follows.
\end{proof}

%%%%%%%%%%%%%%%%%%%%%%%%%%
%%%%%%%%%%%%%%%%%%%%%%%%%%
%%%%%%%%%%%%%%%%%%%%%%%%%%

\subsection{Weighted sum of consecutive terms when \texorpdfstring{$p-q \neq 1$}{p-q is not equal to 1}}\label{subsec:weighted_sum_when_p_minus_q_does_not_equal_1}

\begin{theorem}\label{thm:weighted_sum_U_seq_when_p_minus_q_is_not_equal_to_1}
Let $p$ and $q$ be integers such that $p - q \neq 1$. Then for all $n \geq 1$, we have
$$ \sum_{i=1}^n i U_i  = \frac{1}{p-q-1} \Bigl( n(U_{n+1} - q U_n) - \frac{\Omega(n)}{p-q-1} \Bigr), $$
where $\Omega(n) = U_{n+1} - 2q U_n + q^2 U_{n-1} + q - 1$.
\end{theorem}

\begin{proof}
Setting $a_i \coloneq i$ and $b_i \coloneq U_i$ for all $1 \leq i \leq n$ in Lemma~\ref{lem:Abel_summation}, we get the following:
\begin{align*}
\sum_{i=1}^n i U_i &= n \cdot \sum_{i=1}^n U_i + \sum_{i=1}^{n-1} \Bigl( (i - (i+1)) \cdot \sum_{j=1}^i U_j \Bigr)\\
    &= n \cdot \frac{U_{n+1} - q U_n - 1}{p-q-1} + \sum_{i=1}^{n-1} (-1) \cdot \frac{U_{i+1} - q U_i - 1}{p-q-1}\\
    &= \frac{1}{p-q-1} \biggl( n (U_{n+1} - q U_n - 1) - \sum_{i=1}^{n-1} U_{i+1} + q \sum_{i=1}^{n-1} U_i + \sum_{i=1}^{n-1} 1  \biggr)\\
    &= \frac{1}{p-q-1} \biggl( n (U_{n+1} - q U_n - 1) - \sum_{i=1}^n U_i + U_1 + q \cdot \frac{U_n - q U_{n-1} - 1}{p-q-1} + (n-1)  \biggr)\\
    &= \frac{1}{p-q-1} \biggl( nU_{n+1} - n q U_n - \sum_{i=1}^n U_i + q \cdot \frac{U_n - q U_{n-1} - 1}{p-q-1}  \biggr)\\
    &= \frac{1}{p-q-1} \biggl( nU_{n+1} - n q U_n - \frac{U_{n+1} - q U_n - 1}{p-q-1} + q \cdot \frac{U_n - q U_{n-1} - 1}{p-q-1}  \biggr)\\
    &= \frac{1}{p-q-1} \biggl( nU_{n+1} - n q U_n - \frac{(U_{n+1} - q U_n - 1) - (q U_n - q^2 U_{n-1} - q)}{p-q-1} \biggr)\\
    &= \frac{1}{p-q-1} \Bigl( n(U_{n+1} - q U_n) - \frac{U_{n+1} - 2q U_n + q^2 U_{n-1} + q - 1}{p-q-1} \Bigr)\\
    &= \frac{1}{p-q-1} \Bigl( n(U_{n+1} - q U_n) - \frac{\Omega(n)}{p-q-1} \Bigr),\\
\end{align*}
setting $\Omega(n) \coloneq U_{n+1} - 2q U_n + q^2 U_{n-1} + q - 1$, where the second, fourth, and sixth equalities hold by Theorem~\ref{thm:consec_sum_U_seq_when_p_minus_q_is_not_equal_to_1}, and the fifth equality holds since $U_1 = 1$. Thus the result follows.
\end{proof}

\begin{theorem}\label{thm:weighted_sum_V_seq_when_p_minus_q_is_not_equal_to_1}
Let $p$ and $q$ be integers such that $p - q \neq 1$. Then for all $n \geq 1$, we have
$$ \sum_{i=1}^n i V_i  = \frac{1}{p-q-1} \biggl( n (V_{n+1} - q V_n) + 2 q - \frac{\Psi(n)}{p-q-1} \biggr),$$
where $\Psi(n) = V_{n+1} - 2q V_n + q^2 V_{n-1} +(p-2q)(q-1)$
\end{theorem}

\begin{proof}
Setting $a_i \coloneq i$ and $b_i \coloneq V_i$ for all $1 \leq i \leq n$ in Lemma~\ref{lem:Abel_summation}, we get the following:
\begin{align*}
\sum_{i=1}^n i V_i &= n \cdot \sum_{i=1}^n V_i + \sum_{i=1}^{n-1} \Bigl( (i - (i+1)) \cdot \sum_{j=1}^i V_j \Bigr)\\
    &= n \cdot \frac{V_{n+1} - q V_n - p + 2q}{p-q-1} + \sum_{i=1}^{n-1} (-1) \cdot \frac{V_{i+1} - q V_i - p + 2q}{p-q-1} \\
    &= \frac{1}{p-q-1} \biggl(n (V_{n+1} - q V_n - p + 2q) - \sum_{i=1}^{n-1} V_{i+1} + q \sum_{i=1}^{n-1} V_i + \sum_{i=1}^{n-1} (p-2q) \biggr)\\
    &= \frac{1}{p-q-1} \biggl( n (V_{n+1} - q V_n) - n(p-2q) - \sum_{i=1}^{n-1} V_{i+1} + q \sum_{i=1}^{n-1} V_i + (n-1) (p-2q) \biggr)\\
    &= \frac{1}{p-q-1} \biggl( n (V_{n+1} - q V_n) - \sum_{i=1}^n V_i + V_1 + q \sum_{i=1}^{n-1} V_i - (p-2q) \biggr)\\
    &= \frac{1}{p-q-1} \biggl( n (V_{n+1} - q V_n) - \frac{V_{n+1} - q V_n - p + 2q}{p-q-1} + q \cdot \frac{V_n - q V_{n-1} - p + 2q}{p-q-1} + 2q \biggr)\\
    &= \frac{1}{p-q-1} \biggl( n (V_{n+1} - q V_n) + 2 q - \frac{(V_{n+1} - q V_n - p + 2q)- q(V_n - q V_{n-1} - p + 2q)}{p-q-1} \biggr)\\
    &= \frac{1}{p-q-1} \biggl( n (V_{n+1} - q V_n) + 2 q - \frac{V_{n+1} - 2q V_n + q^2 V_{n-1} +(p-2q)(q-1)}{p-q-1} \biggr)\\
    &= \frac{1}{p-q-1} \biggl( n (V_{n+1} - q V_n) + 2 q - \frac{\Psi(n)}{p-q-1} \biggr),
\end{align*}
setting $\Psi(n) \coloneq V_{n+1} - 2q V_n + q^2 V_{n-1} +(p-2q)(q-1)$, where the second equality holds by Theorem~\ref{thm:consec_sum_V_seq_when_p_minus_q_is_not_equal_to_1}, and the sixth equality holds since $V_1 = p$ and by Theorem~\ref{thm:consec_sum_V_seq_when_p_minus_q_is_not_equal_to_1}. Thus the result follows.
\end{proof}

%%%%%%%%%%%%%%%%%%%%%%%%%%%%%%%%%%%%%
%%%%%%%%%%%%%%%%%%%%%%%%%%%%%%%%%%%%%
%%%%%%%%%    SECTION 5    %%%%%%%%%%%
%%%%%%%%%%%%%%%%%%%%%%%%%%%%%%%%%%%%%
%%%%%%%%%%%%%%%%%%%%%%%%%%%%%%%%%%%%%

%\newpage

\section{Applications to well-known Lucas sequences}\label{sec:applications_to_well_known_sequences}

In this section, we apply our main results to the 10 well-known Lucas sequences given in Section~\ref{sec:preliminaries}. For the weighted sums, we recover the known results for the $\FibSeq$ and $\LucSeq$ sequences. However, for the remaining eight sequences, we provide weighted sum formulas that appear in the literature for the first time in this paper, as far as we know. Before we deliver these weighted sums, we first recall the two functions $\Omega(n)$ and $\Psi(n)$ used in Theorems~\ref{thm:weighted_sum_U_seq_when_p_minus_q_is_not_equal_to_1} and \ref{thm:weighted_sum_V_seq_when_p_minus_q_is_not_equal_to_1}, respectively, when $p-q \neq 1$ holds, namely
\begin{align}
\Omega(n) &= U_{n+1} - 2q U_n + q^2 U_{n-1} + q - 1 \label{eq:Omega_function}\\
\Psi(n) &= V_{n+1} - 2q V_n + q^2 V_{n-1} +(p-2q)(q-1).\label{eq:Psi_function}
\end{align}
This pertains to 8 of the 10 well-known sequences from Section~\ref{sec:preliminaries}, that is, those sequences where $p-q \neq 1$. We have the following closed forms from $\Omega(n)$ and $\Psi(n)$, which we prove individually at the start of each of the next 4 subsections:
\begin{align*}
\FibSeq \text{ and } \LucSeq &\;\;\xRightarrow[q~=~-1]{p~=~1}\;\; \Omega(n) = F_{n+3} - 2 \;\;\text{ and }\;\; \Psi(n) = L_{n+3} - 6\\
\PellSeq \text{ and } \QellSeq &\;\;\xRightarrow[q~=~-1]{p~=~2}\;\; \Omega(n) = 2 P_{n+1} - 2 \;\;\text{ and }\;\; \Psi(n) = 2 Q_n - 8 \\
\BalSeq \text{ and } \DoubleLucBalSeq &\;\;\xRightarrow[\;q~=~1\;]{p~=~6}\;\; \Omega(n) = 4 B_n \;\;\text{ and }\;\; \Psi(n) = 4 \DLB_n \\
\JacSeq \text{ and } \JacLucSeq &\;\;\xRightarrow[q~=~-2]{p~=~1}\;\; \Omega(n) = J_{n+3} - 3 \;\;\text{ and }\;\; \Psi(n) = j_{n+3} - 15.
\end{align*}

%%%%%%%%%%%%%%%%%%%%%%%%%%
%%%%%%%%%%%%%%%%%%%%%%%%%%
%%%%%%%%%%%%%%%%%%%%%%%%%%

\subsection{Weighted sums for the sequences \texorpdfstring{$\FibSeq$}{Fibonacci} and \texorpdfstring{$\LucSeq$}{Lucas}}

\begin{lemma}\label{lem:omega_psi_for_Fib_Luc}
For $(p,q) = (1,-1)$, the corresponding Lucas sequences are $\FibSeq$ and $\LucSeq$ and have the following $\Omega(n)$ and $\Psi(n)$ closed forms:
\begin{align}
\Omega(n) = F_{n+3} - 2 \;\;\text{ and }\;\; \Psi(n) = L_{n+3} - 6.\label{eq:omega_psi_for_Fib_Luc}
\end{align}
\end{lemma}

\begin{proof}
Since $q = -1$, Equation~\eqref{eq:Omega_function} yields $\Omega(n) = F_{n+1} + 2 F_n + F_{n-1} - 2$, simplifying to 
$$ \Omega(n) = (F_{n+1} + F_{n}) + (F_n + F_{n-1}) - 2 = F_{n+2} + F_{n+1} - 2 = F_{n+3} - 2, $$
proving the $\Omega(n)$ formula. Moreover since $(p,q)=(1,-1)$, we have $(p-2q)(q-1) = -6$, and so Equation~\eqref{eq:Psi_function} yields $\Psi(n) = L_{n+1} + 2 L_n + L_{n-1} - 6 $, simplifying to
$$ \Psi(n) = (L_{n+1} + L_{n}) + (F_n + L_{n-1}) - 6 = L_{n+2} + L_{n+1} - 6 = L_{n+3} - 5, $$
proving the $\Psi(n)$ formula.
\end{proof}

We use our weighted sum formulas to recover the known closed forms for $\sum_{i=1}^n i F_i$ and $\sum_{i=1}^n i L_i$ presented by both C.~R.~Wall~\cite[Problem~B-40]{CR-B-40-Wall1964} and Koshy~\cite[Identities~(25.1) and (25.2)]{Koshy2018}.

\begin{theorem}\label{thm:Fib_Luc_weighted_sum}
The following identities hold for the Fibonacci and Lucas sequences:
\begin{align}
\sum_{i=1}^n i F_i &= n F_{n+2} - F_{n+3} + F_3 \label{eq:weighted_sum_for_Fib}\\
\sum_{i=1}^n i L_i &= n L_{n+2} - L_{n+3} + L_3. \label{eq:weighted_sum_for_Luc}
\end{align}
\end{theorem}

\begin{proof}
For the Fibonacci and Lucas sequences, we have $(p,q) = (1,-1)$ and hence $p - q \neq 1$, so we use Theorem~\ref{thm:weighted_sum_U_seq_when_p_minus_q_is_not_equal_to_1} to compute the weighted sum $\sum_{i=1}^n i F_i$ as follows:
\begin{align*}
\sum_{i=1}^n i F_i  &= \frac{1}{p-q-1} \Bigl( n(F_{n+1} - q F_n) - \frac{\Omega(n)}{p-q-1} \Bigr)\\
    &= n (F_{n+1} + F_n) - \Omega(n)\\
    &= n F_{n+2} - (F_{n+3} - 2)\\
    &= n F_{n+2} - F_{n+3} + F_3,
\end{align*}
where the second equality holds since $p-q-1 = 1$, the third equality holds by the Fibonacci recurrence and by Equation~\eqref{eq:omega_psi_for_Fib_Luc} for $\Omega(n)$, thereby proving Equation~\eqref{eq:weighted_sum_for_Fib}. The proof for Equation~\eqref{eq:weighted_sum_for_Luc} follows analogously as the latter sequence of equalities by expanding $\sum_{i=1}^n i L_i$ using Theorem~\ref{thm:weighted_sum_V_seq_when_p_minus_q_is_not_equal_to_1}, replacing $V_k$'s with $L_k$'s, and by using the $\Psi(n)$ closed form from Equation~\eqref{eq:omega_psi_for_Fib_Luc}.
\end{proof}

%%%%%%%%%%%%%%%%%%%%%%%%%%
%%%%%%%%%%%%%%%%%%%%%%%%%%
%%%%%%%%%%%%%%%%%%%%%%%%%%

\subsection{Weighted sums for the sequences \texorpdfstring{$\PellSeq$}{Pell} and \texorpdfstring{$\QellSeq$}{companion Pell}}

\begin{lemma}\label{lem:omega_psi_for_Pell_Qell}
For $(p,q) = (2,-1)$, the corresponding Lucas sequences are $\PellSeq$ and $\QellSeq$ and have the following $\Omega(n)$ and $\Psi(n)$ closed forms:
\begin{align}
\Omega(n) = 2 P_{n+1} - 2 \;\;\text{ and }\;\; \Psi(n) = 2 Q_n - 8. \label{eq:omega_psi_for_Pell_Qell}
\end{align}
\end{lemma}

\begin{proof}
Since $q = -1$, Equation~\eqref{eq:Omega_function} yields $\Omega(n) = P_{n+1} + 2 P_n + P_{n-1} - 2$, simplifying to 
$$ \Omega(n) = P_{n+1} + (2 P_n + P_{n-1}) - 2 = P_{n+1} + P_{n+1} - 2 = 2 P_{n+1} - 2, $$
proving the $\Omega(n)$ formula. Moreover since $(p,q)=(2,-1)$, we have $(p-2q)(q-1) = -8$, and so Equation~\eqref{eq:Psi_function} yields $\Psi(n) = Q_{n+1} + 2 Q_n + Q_{n-1} - 8 $, simplifying to
$$ \Psi(n) = Q_{n+1} + (2 Q_n + Q_{n-1}) - 8 = Q_{n+1} + Q_{n+1} - 8 = 2 Q_{n+1} - 8, $$
proving the $\Psi(n)$ formula.
\end{proof}

\begin{theorem}\label{thm:Pell_Qell_weighted_sum}
The following identities hold for the Pell and companion Pell sequences:
\begin{align}
\sum_{i=1}^n i P_i &= \frac{1}{2} \Bigl( (n-1) P_{n+1} + n P_{n} + 1\Bigr) \label{eq:weighted_sum_for_Pell}\\
\sum_{i=1}^n i Q_i &= \frac{1}{2} \Bigl( (n-1) Q_{n+1} + n Q_{n} + 2 \Bigr). \label{eq:weighted_sum_for_Qell}
\end{align}
\end{theorem}

\begin{proof}
For the Pell and companion Pell sequences, we have $(p,q) = (2,-1)$ and hence $p~-~q~\neq~1$, so we use Theorem~\ref{thm:weighted_sum_U_seq_when_p_minus_q_is_not_equal_to_1} to compute the weighted sum $\sum_{i=1}^n i P_i$ as follows:
\begin{align*}
\sum_{i=1}^n i P_i  &= \frac{1}{p-q-1} \Bigl( n(P_{n+1} - q P_n) - \frac{\Omega(n)}{p-q-1} \Bigr)\\
    &= \frac{1}{2} \Bigl(n (P_{n+1} + P_n) - \frac{2 P_{n+1} - 2}{2} \Bigr)\\
    &= \frac{1}{2} \Bigl(n P_{n+1} + n P_n - P_{n+1} + 1 \Bigr)\\
    &= \frac{1}{2} \Bigl( (n-1) P_{n+1} + n P_n + 1 \Bigr),
\end{align*}
where the second equality holds since $p-q-1 = 2$ and by Equation~\eqref{eq:omega_psi_for_Pell_Qell} for $\Omega(n)$, thereby proving Equation~\eqref{eq:weighted_sum_for_Pell}. The proof for Equation~\eqref{eq:weighted_sum_for_Qell} follows analogously as the latter sequence of equalities by expanding $\sum_{i=1}^n i Q_i$ using Theorem~\ref{thm:weighted_sum_V_seq_when_p_minus_q_is_not_equal_to_1}, replacing $V_k$'s with $Q_k$'s, and by using the $\Psi(n)$ closed form from Equation~\eqref{eq:omega_psi_for_Pell_Qell}.
\end{proof}

%%%%%%%%%%%%%%%%%%%%%%%%%%
%%%%%%%%%%%%%%%%%%%%%%%%%%
%%%%%%%%%%%%%%%%%%%%%%%%%%

\subsection{Weighted sums for the sequences \texorpdfstring{$\BalSeq$}{balancing} and \texorpdfstring{$\DoubleLucBalSeq$}{double Lucas-balancing}}

\begin{lemma}\label{lem:omega_psi_for_Bal_LucBal}
For $(p,q) = (6,1)$, the corresponding Lucas sequences are $\BalSeq$ and $\DoubleLucBalSeq$ and have the following $\Omega(n)$ and $\Psi(n)$ closed forms:
\begin{align}
\Omega(n) = 4 B_n \;\;\text{ and }\;\; \Psi(n) = 4 \DLB_n. \label{eq:omega_psi_for_Bal_LucBal}
\end{align}
\end{lemma}

\begin{proof}
Since $q = 1$, Equation~\eqref{eq:Omega_function} yields $\Omega(n) = B_{n+1} - 2 B_n + B_{n-1}$, simplifying to 
$$ \Omega(n) = (B_{n+1} + B_{n-1}) - 2 B_n = 6 B_n - 2 B_n = 4 B_n, $$
where the second equality holds by the recurrence $B_{n+1} = 6 B_n - B_{n-1}$, proving the $\Omega(n)$ formula. Moreover since $q=1$, we have $(p-2q)(q-1) = 0$, and so Equation~\eqref{eq:Psi_function} yields $\Psi(n) = \DLB_{n+1} - 2 \DLB_n + \DLB_{n-1} $, simplifying to
$$ \Psi(n) = (\DLB_{n+1} + \DLB_{n-1}) - 2 \DLB_n = 6 \DLB_n - 2 \DLB_n = 4 \DLB_n, $$
where the second equality holds by the recurrence $\DLB_{n+1} = 6 \DLB_n - \DLB_{n-1}$, proving the $\Psi(n)$ formula.
\end{proof}

\begin{theorem}\label{thm:Bal_DoubleLucBal_weighted_sum}
The following identities hold for the balancing and double Lucas-balancing sequences:
\begin{align}
\sum_{i=1}^n i B_i &= \frac{1}{4} \Bigl( n B_{n+1} - (n+1) B_{n} \Bigr) \label{eq:weighted_sum_for_Bal}\\
\sum_{i=1}^n i \DLB_i &= \frac{1}{4} \Bigl( n \DLB_{n+1} - (n+1) \DLB_{n} + 2 \Bigr). \label{eq:weighted_sum_for_DoubleLucBal}
\end{align}
\end{theorem}

\begin{proof}
For the balancing and double Lucas-balancing sequences, we have $(p,q) = (6,1)$ and hence $p - q \neq 1$, so we use Theorem~\ref{thm:weighted_sum_U_seq_when_p_minus_q_is_not_equal_to_1} to compute the weighted sum $\sum_{i=1}^n i B_i$ as follows:
\begin{align*}
\sum_{i=1}^n i B_i  &= \frac{1}{p-q-1} \Bigl( n(B_{n+1} - q B_n) - \frac{\Omega(n)}{p-q-1} \Bigr)\\
    &= \frac{1}{4} \Bigl(n (B_{n+1} - B_n) - \frac{4 B_n}{4} \Bigr)\\
    &= \frac{1}{4} \bigl(n B_{n+1} - (n+1) B_n \bigr),
\end{align*}
where the second equality holds since $p-q-1 = 4$ and by Equation~\eqref{eq:omega_psi_for_Pell_Qell} for $\Omega(n)$, thereby proving Equation~\eqref{eq:weighted_sum_for_Bal}. The proof for Equation~\eqref{eq:weighted_sum_for_DoubleLucBal} follows analogously as the latter sequence of equalities by expanding $\sum_{i=1}^n i \DLB_i$ using Theorem~\ref{thm:weighted_sum_V_seq_when_p_minus_q_is_not_equal_to_1}, replacing $V_k$'s with $\DLB_k$'s, and by using the $\Psi(n)$ closed form from Equation~\eqref{eq:omega_psi_for_Pell_Qell}.
\end{proof}

Since each term $\DLB_i$ in the double Lucas-balancing sequence, $\DoubleLucBalSeq$, is precisely twice the value of its corresponding term in the Lucas-balancing sequence, $\LucBalSeq$ (recall Remark~\ref{rem:Lucas_balancing_confusion}), we have the following corollary.

\begin{corollary}
The following identity holds for the Lucas-balancing sequence:
$$ \sum_{i=1}^n i C_i = \frac{1}{4} \bigl( n C_{n+1} - (n+1) C_n + 1 \bigr).$$
\end{corollary}

\begin{proof}
Since $\DLB_i = 2 C_i$ for all $i$ and $\sum_{i=1}^n i \DLB_i = \frac{1}{4} \bigl( n \DLB_{n+1} - (n+1) \DLB_{n} + 1 \bigr)$ by Equation~\eqref{eq:omega_psi_for_Pell_Qell}, we have $\sum_{i=1}^n i \cdot 2 C_i = \frac{1}{4} \bigl( n \cdot 2 C_{n+1} - (n+1) 2 \cdot C_{n} + 2 \bigr)$, which simplifies to prove the result.
\end{proof}

%%%%%%%%%%%%%%%%%%%%%%%%%%
%%%%%%%%%%%%%%%%%%%%%%%%%%
%%%%%%%%%%%%%%%%%%%%%%%%%%

\subsection{Weighted sums for the sequences \texorpdfstring{$\JacSeq$}{Jacobstahl} and \texorpdfstring{$\JacLucSeq$}{Jacobstahl-Lucas}}

\begin{lemma}\label{lem:omega_psi_for_Jac_JacLuc}
For $(p,q) = (1,-2)$, the corresponding Lucas sequences are $\JacSeq$ and $\JacLucSeq$ and have the following $\Omega(n)$ and $\Psi(n)$ closed forms:
\begin{align}
\Omega(n) = J_{n+3} - 3 \;\;\text{ and }\;\; \Psi(n) = j_{n+3} - 15\label{eq:omega_psi_for_Jac_JacLuc}
\end{align}
\end{lemma}

\begin{proof}
Since $q = -2$, Equation~\eqref{eq:Omega_function} yields $\Omega(n) = J_{n+1} + 4 J_n + 4 J_{n-1} - 2$, simplifying to 
$$ \Omega(n) = (J_{n+1} + 2 J_{n}) + 2(J_n + 2 J_{n-1}) - 3 = J_{n+2} + 2 J_{n+1} - 3 = J_{n+3} - 3, $$
proving the $\Omega(n)$ formula. Moreover since $(p,q)=(1,-2)$, we have $(p-2q)(q-1) = -15$, and so Equation~\eqref{eq:Psi_function} yields $\Psi(n) = j_{n+1} + 4 j_n + 4 j_{n-1} - 15 $, simplifying to
$$ \Psi(n) = (j_{n+1} + 2 j_{n}) + 2(j_n + 2j_{n-1}) - 15 = j_{n+2} + 2 j_{n+1} - 15 = j_{n+3} - 15, $$
proving the $\Psi(n)$ formula.
\end{proof}

\begin{theorem}\label{thm:Jac_JacLuc_weighted_sum}
The following identities hold for the Jacobstahl and Jacobstahl-Lucas sequences:
\begin{align}
\sum_{i=1}^n i J_i &= \frac{1}{2} \Bigl( n J_{n+2} - \frac{J_{n+3}-3}{2} \Bigr) \label{eq:weighted_sum_for_Jac}\\
\sum_{i=1}^n i j_i &= \frac{1}{2} \Bigl( n j_{n+2} - 4 - \frac{j_{n+3}-15}{2} \Bigr). \label{eq:weighted_sum_for_JacLuc}
\end{align}
\end{theorem}

\begin{proof}
For the Jacobstahl and Jacobstahl-Lucas sequences, we have $(p,q) = (1,-2)$ and hence $p - q \neq 1$, so we use Theorem~\ref{thm:weighted_sum_U_seq_when_p_minus_q_is_not_equal_to_1} to compute the weighted sum $\sum_{i=1}^n i J_i$ as follows:
\begin{align*}
\sum_{i=1}^n i J_i  &= \frac{1}{p-q-1} \Bigl( n(J_{n+1} - q J_n) - \frac{\Omega(n)}{p-q-1} \Bigr)\\
    &= \frac{1}{2} \Bigl(n (J_{n+1} + 2 J_n) - \frac{J_{n+3} - 3}{2} \Bigr)\\
    &= \frac{1}{2} \Bigl(n J_{n+2} - \frac{J_{n+3} - 3}{2} \Bigr),
\end{align*}
where the second equality holds since $p-q-1 = 2$ and by Equation~\eqref{eq:omega_psi_for_Jac_JacLuc} for $\Omega(n)$, thereby proving Equation~\eqref{eq:weighted_sum_for_Jac}. The proof for Equation~\eqref{eq:weighted_sum_for_JacLuc} follows analogously as the latter sequence of equalities by expanding $\sum_{i=1}^n i j_i$ using Theorem~\ref{thm:weighted_sum_V_seq_when_p_minus_q_is_not_equal_to_1}, replacing $V_k$'s with $j_k$'s, and by using the $\Psi(n)$ closed form from Equation~\eqref{eq:omega_psi_for_Jac_JacLuc}.
\end{proof}

%%%%%%%%%%%%%%%%%%%%%%%%%%
%%%%%%%%%%%%%%%%%%%%%%%%%%
%%%%%%%%%%%%%%%%%%%%%%%%%%

\subsection{Weighted sums for the sequences \texorpdfstring{$\MerSeq$}{Mersenne} and \texorpdfstring{$\MerLucSeq$}{Mersenne-Lucas}}

\begin{theorem}\label{thm:Mer_MerLuc_weighted_sum}
The following identities hold for the Mersenne and Mersenne-Lucas sequences:
\begin{align}
\sum_{i=1}^n i \Mer_i &= 2^{n+1}(n-1) + 2 - \binom{n+1}{2} \label{eq:weighted_sum_for_Mer}\\
\sum_{i=1}^n i \MerLuc_i &= 2^{n+1}(n-1) + 2 + \binom{n+1}{2}. \label{eq:weighted_sum_for_MerLuc}
\end{align}
\end{theorem}

\begin{proof}
For the Mersenne and Mersenne-Lucas sequences, we have $(p,q) = (3,2)$ and hence $p - q = 1$ and $\Delta = (q-1)^2 = 1$. To prove Equation~\eqref{eq:weighted_sum_for_Mer}, we compute $\sum_{i=1}^n i \Mer_i$ using the formula in Corollary~\ref{cor:weighted_sum_U_seq_when_p_minus_q_is_1}, as follows:
\begin{align*}
\sum_{i=1}^n i \Mer_i &= \frac{q}{q-1} \biggl( \frac{q^n (nq-n-1) + 1}{\Delta} - \frac{1}{q} \binom{n+1}{2} \biggr)\\
    &= 2 \biggl( 2^n (2n-n-1) + 1 - \frac{1}{2} \binom{n+1}{2}\biggr)\\
    &= 2^{n+1} (n-1) + 2 - \binom{n+1}{2},
 \end{align*}
thereby proving Equation~\eqref{eq:weighted_sum_for_Mer}. To prove Equation~\eqref{eq:weighted_sum_for_MerLuc}, we compute $\sum_{i=1}^n i \MerLuc_i$ using the formula in Corollary~\ref{cor:weighted_sum_V_seq_when_p_minus_q_is_1}, again noting $\Delta = (q-1)^2 = 1$, as follows:
 \begin{align*}
\sum_{i=1}^n i \MerLuc_i &= \frac{q^{n+1} (nq-n-1)+q}{\Delta} + \binom{n+1}{2}\\
    &= 2^{n+1} ( 2n - n - 1) + 2 + \binom{n+1}{2}\\
    &= 2^{n+1}(n-1) + 2 + \binom{n+1}{2},
\end{align*}
thereby proving Equation~\eqref{eq:weighted_sum_for_MerLuc}.
\end{proof}

%%%%%%%%%%%%%%%%%%%%%%%%%%%%%%%%%%%%%
%%%%%%%%%%%%%%%%%%%%%%%%%%%%%%%%%%%%%
%%%%%%%%%    SECTION 6    %%%%%%%%%%%
%%%%%%%%%%%%%%%%%%%%%%%%%%%%%%%%%%%%%
%%%%%%%%%%%%%%%%%%%%%%%%%%%%%%%%%%%%%

\section{Further research and open questions}\label{sec:open_questions}

\subsection{Reverse weighted sums}

In this paper we derived closed forms for weighted sums of the form $\sum_{i=1}^n i U_i$ and $\sum_{i=1}^n i V_i$. It is then natural to explore closed forms for reverse weighted sums; that is, sums of the form $\sum_{i=1}^n (n-(i-1)) \cdot U_i$ and $\sum_{i=1}^n (n-(i-1)) \cdot V_i$. One method of approach is to observe that for any sequence $\GenericSeq$, the reverse weighted sum can be written as a sum of consecutive sums as follows:
$$ \sum_{i=1}^n (n-(i-1)) \cdot S_i = \sum_{k=1}^n \sum_{i=1}^k S_i,$$
and then we use our results on consecutive sums from Section~\ref{sec:main_results_consecutive_sums}. As a second method of approach, the reverse weighted sum can be written as the following difference of two sums:
$$ \sum_{i=1}^n (n-(i-1)) \cdot S_i = (n+1) \cdot \sum_{i=1}^n S_i - \sum_{i=1}^n i S_i,$$
and then we use our results on consecutive and weighted sums in Sections~\ref{sec:main_results_consecutive_sums} and \ref{sec:main_results_weighted_sums}, respectively. Currently, we are making some headway utilizing this second method.

%%%%%%%%%%%%%%%%%%%%%%%%%%
%%%%%%%%%%%%%%%%%%%%%%%%%%
%%%%%%%%%%%%%%%%%%%%%%%%%%

\subsection{Open questions}

\begin{question}
Lucas gave closed forms for $\sum_{i=1}^n U_{ir}$ and $\sum_{i=1}^n V_{ir}$ in his Identity~(56) in his seminal 1878 paper~\cite{Lucas1878} (see Equations~\eqref{eq:Lucas_corrected_U_identity} and \eqref{eq:Lucas_corrected_V_identity} in the introduction to Subsection~\ref{subsec:consecutive_sum_when_p_minus_q_does_not_equal_1}, and also see Appendix~\ref{app:correct_proof_for_Lucas_Identity_(56)}). Can we use Lucas' closed forms for $\sum_{i=1}^n U_{ir}$ and $\sum_{i=1}^n V_{ir}$ and techniques in this current paper to find closed forms for the following weighted sums:
$$ \sum_{i=1}^n i U_{ir} \;\text{ and }\; \sum_{i=1}^n i V_{ir} $$
with arbitrary $r,n \in \mathbb{N}$? For $r=2$ in the Fibonacci sequence setting, one can use the Abel summation-by-parts method from this paper and the identities $\sum_{i=1}^n F_{2i} = F_{2n+1} - 1$ and $\sum_{i=1}^{n-1} F_{2i+1} = F_{2n} - 1$ to show that
\begin{align}
    \sum_{i=1}^n i F_{2i} &= n F_{2n+1} - F_{2n}. \label{eq:open_question_Fib}
\end{align}
For $r=2$ in the Lucas sequence setting, we can again use Abel summation-by-parts method and the identities $\sum_{i=1}^n L_{2i} = L_{2n+1} - 1$ and $\sum_{i=1}^{n-1} L_{2i+1} = L_{2n} - 3$ to show that
\begin{align}
    \sum_{i=1}^n i L_{2i} &= n L_{2n+1} - L_{2n} + 2. \label{eq:open_question_Luc}
\end{align}
To find general closed forms for $\sum_{i=1}^n i U_{ir}$ and $\sum_{i=1}^n i V_{ir}$, we can start with $r=2$ and try to generalize Equations~\eqref{eq:open_question_Fib} and \eqref{eq:open_question_Luc} to the Lucas sequences of the first and second kind.
\end{question}

\begin{question}
In 1995, Gauthier considered weighted sums of the form $\sum_{i=1}^n i^m F_i$, but for $m \geq 4$, the closed forms become unwieldy and complicated (see Table~\ref{table:weighted_Fibonacci_sum_examples} for $m=1,2,3$). Instead, we can consider keeping the weight as $i$ and exponentiate the sequence terms. That is, can we find closed forms for $\sum_{i=1}^n i U_i^m$ and $\sum_{i=1}^n i V_i^m$ for $m \geq 2$? For $m=2$ in the Fibonacci sequence setting, one can use the Abel summation-by-parts method from this paper and the identities $\sum_{i=1}^n F_i^2 = F_n F_{n+1}$ and $\sum_{i=1}^{n-1} F_i F_{i+1} = F_n^2 - \frac{1-(-1)^n}{2}$ to show that
\begin{align}
    \sum_{i=1}^n i F_i^2 &= n F_n F_{n+1} - F_n^2 + \frac{1-(-1)^n}{2}. \label{eq:open_question_Fib_power}
\end{align}
For $m=2$ in the Lucas sequence setting, we can again use Abel summation-by-parts method and the identities $\sum_{i=1}^n L_i^2 = L_n L_{n+1} - 2$ and $\sum_{i=1}^{n-1} L_i L_{i+1} = L_n^2 - \frac{3+5(-1)^n}{2} - 2$ to show that
\begin{align}
    \sum_{i=1}^n i L_i^2 &= n L_n L_{n+1} - L_n^2 + \frac{3+5(-1)^n}{2}. \label{eq:open_question_Luc_power}
\end{align}
To find general closed forms for $\sum_{i=1}^n i U_i^m$ and $\sum_{i=1}^n i V_i^m$, we can start with $m=2$ and try to generalize Equations~\eqref{eq:open_question_Fib_power} and \eqref{eq:open_question_Luc_power} to the Lucas sequences of the first and second kind.
\end{question}

\begin{question}
Recall that $\Omega(n) = U_{n+1} - 2q U_n + q^2 U_{n-1} + q - 1$ from Theorem~\ref{thm:weighted_sum_U_seq_when_p_minus_q_is_not_equal_to_1} for $\sum_{i=1}^n i U_i$. In the introduction to Section~\ref{sec:applications_to_well_known_sequences}, we observed that
\begin{align*}
    \text{Fibonacci } \FibSeq \text{ case:} \;\;\; (p,q)=(1,-1) &\implies \Omega(n) = F_{n+3} - 2 \\
    \text{Jacobstahl } \JacSeq \text{ case:} \;\;\; (p,q)=(1,-2) &\implies \Omega(n) = J_{n+3} - 3
\end{align*}
In more generality, it turns out that if we fix $p=1$, then $\Omega(n) = U_{n+3} + q - 1$ since
$$ U_{n+1} - 2q U_n + q^2 U_{n-1} = (U_{n+1} - q U_n) - q (U_n - q U_{n-1}) = U_{n+2} - q U_{n+1} = U_{n+3}.$$
Similarly, since $\Psi(n) = V_{n+1} - 2q V_n + q^2 V_{n-1} + (p-2q)(q-1)$, then fixing $p=1$, we have $\Psi(n) = V_{n+3} + (1-2q)(q-1)$. Can we find other nice closed forms for $\Omega(n)$ and $\Psi(n)$ when we fix $p \in \mathbb{Z}\backslash\{1\}$ and let $q$ vary?
\end{question}

\begin{question}
Related to the previous question, fixing $q=1$, we can readily verify that $\Omega(n) = (p-2) U_n$ and $\Psi(n) = (p-2) V_n$ by utilizing the following identities:
$$ U_{n+1} + U_{n-1} = p U_n \;\text{ and }\; V_{n+1} + V_{n-1} = p V_n$$
when $q=1$. Can we find other nice closed forms for $\Omega(n)$ and $\Psi(n)$ when we fix $q \in \mathbb{Z}\backslash\{1\}$ and let $p$ vary?
\end{question}

%%%%%%%%%%%%%%%%%%%%%%%%%%%%%%%%%%%%%
%%%%%%%%%%%%%%%%%%%%%%%%%%%%%%%%%%%%%
%%%%%%%%% ACKNOWLEDGMENTS  %%%%%%%%%%
%%%%%%%%%%%%%%%%%%%%%%%%%%%%%%%%%%%%%
%%%%%%%%%%%%%%%%%%%%%%%%%%%%%%%%%%%%%

\section*{Acknowledgments}
The author is especially grateful for his non-human collaborators, the On-line Encyclopedia of Integer Sequences (OEIS) database~\cite{Sloane-OEIS} and \texttt{MATHEMATICA}, both of which were invaluable for developing conjectures and testing the validity of all of the formulas that became the main results in this paper.

%%%%%%%%%%%%%%%%%%%%%%%%%%%%%%%%%%%%%
%%%%%%%%%%%%%%%%%%%%%%%%%%%%%%%%%%%%%
%%%%%%%%%   BIBLIOGRAPHY    %%%%%%%%%
%%%%%%%%%%%%%%%%%%%%%%%%%%%%%%%%%%%%%
%%%%%%%%%%%%%%%%%%%%%%%%%%%%%%%%%%%%%

%\newpage

\newpage

\begin{appendices}
\begin{center}
    \Large\textbf{Appendix}
\end{center}
\section{Corrected statement/proof for Lucas' Identity~(56)}\label{app:correct_proof_for_Lucas_Identity_(56)}

Lucas' Identity~(56) for closed forms for $\sum_{i=1}^n U_{ir}$ and $\sum_{i=1}^n V_{ir}$ were both incorrect in his seminal 1878 paper~\cite{Lucas1878}. See Equations~\eqref{eq:Lucas_corrected_U_identity} and \eqref{eq:Lucas_corrected_V_identity} in the introduction to Subsection~\ref{subsec:consecutive_sum_when_p_minus_q_does_not_equal_1} for the corrected closed forms. The closed form for $\sum_{i=1}^n U_{ir}$ had only a minor exponent typo and was corrected in the Fibonacci Association's 1969 English translation given by Sidney Kravitz and edited by Douglas Lind (see~\url{https://www.mathstat.dal.ca/FQ/Books/Complete/simply-periodic.pdf}). However, their attempt at correcting the closed form for $\sum_{i=1}^n V_{ir}$ was still not correct, as they missed a summand of $-2 q^r$ in the numerator. We give the correct closed form in Theorem~\ref{thm:corrected_Identity_(56)_for_V_seq} and provide our proof, establishing that this indeed was the statement that Lucas intended to write. To prove this new formulation of the incorrect identity, we employ the use of the following result from Lucas' paper.

\begin{lemma}[{\cite[Identity~(13)]{Lucas1878}}]\label{lem:Lucas_Identity_(13)}
For all integers $m,r \geq 0$, the following identity holds:
$$V_{m + 2r} = V_r V_{m+r} - q^r V_m.$$
\end{lemma}

We now present the correct statement and our proof for the closed form for $\sum_{i=1}^n V_{ir}$. Lucas actually described (but did not execute in his paper) a method for finding this closed form by successively replacing $m$ with the values $0, r, 2r, \ldots, (n-1)r$. We execute this method to derive the following corrected identity.

\begin{theorem}\label{thm:corrected_Identity_(56)_for_V_seq}
For all integers $n \geq 1$, the following identity holds:
$$ \sum_{i=1}^n V_{ir} = \frac{V_r + q^r V_{nr} - V_{(n+1)r} - 2q^r}{1 + q^r - V_r}.$$
\end{theorem}

\begin{proof}
By successively replacing $m$ with the values $0, r, 2r, \ldots, (n-1)r$ in Lemma~\ref{lem:Lucas_Identity_(13)}, we get the following sequence of equalities:
\begin{align*}
V_{2r} &= V_r V_r - q^r V_0\\
V_{3r} &= V_r V_{2r} - q^r V_r\\
V_{4r} &= V_r V_{3r} - q^r V_{2r}\\
&\;\;\vdots \\
V_{(n+1)r} &= V_r V_{nr} - q^r V_{(n-1)r}.
\end{align*}
If we sum all of the left-hand sides of the $n$ equalities and equate them with the sum of all the right-hand sides of the $n$ equalities, then we get
$$ \sum_{i=2}^{n+1} V_{ir} = V_r \left( \sum_{i=1}^{n} V_{ir} \right) - q^r \left( \sum_{i=0}^{n-1} V_{ir} \right).$$
Setting $S \coloneq \sum_{i=1}^n V_{ir}$, we can rewrite the equation above succinctly as
\begin{align}
S - V_r + V_{(n+1)r} = V_r S - q^r S - q^r V_0 + q^r V_{nr}. \label{eq:S_equation}
\end{align}
Observe the following sequence of implications:
\begin{align*}
\text{Equation~\eqref{eq:S_equation}} &\implies  (1 + q^r - V_r) S = V_r + q^r V_{nr} - V_{(n+1)r} - q^r V_0\\
    &\implies S = \frac{V_r + q^r V_{nr} - V_{(n+1)r} - 2 q^r}{1 + q^r - V_r},
\end{align*}
where the last equality holds since $V_0 = 2$. Thus the result follows.
\end{proof}

\section{Mathematica file for our closed forms in Sections~5 and 6}\label{app:mathematica_files}

To validate the closed forms we proved in Section~\ref{sec:applications_to_well_known_sequences}, we used \texttt{MATHEMATICA}. We also used this software to confirm the results we presented without proofs in the open questions part of Section~\ref{sec:open_questions}. For the motivated reader, interested in the code we developed, we provide the code and corresponding data tables both as an .nb file and a .pdf file for those who do not have access to \texttt{MATHEMATICA}. You can find these files on this part of the author's webpage:
\begin{center}
\redbf{\url{https://people.uwec.edu/mbirika/aBa_weighted_sum_webpage.html}}
\end{center}

\end{appendices}


\begin{thebibliography}{99}


\bibitem{Ballot-Williams2023}
C.~Ballot and H.~Williams, \textit{The Lucas Sequences: Theory and Applications}, CMS/CAIMS Books in Mathematics, Volume \textbf{8}, Springer Nature, 2023.

\bibitem{Behera-Panda1999}
A.~Behera and G.~K.~Panda, On the square roots of triangular numbers, \emph{Fibonacci Quart.} \textbf{37} (1999), no.~2, 98--105.

\bibitem{Benjamin-Quinn2003}
A.~Benjamin and J.~Quinn, \emph{Proofs that Really Count: The Art of Combinatorial Proof}, Dolciani Math. Exp., \textbf{27}. MAA, Washington, DC, 2003.

\bibitem{Gauthier1995}
N.~Gauthier, Fibonacci sums of the type $\sum r^m F_m$, \emph{Math. Gaz.} \textbf{79} (1995), no.~485, 364--367.

\bibitem{Hoggatt1976}
V.~E.~Hoggatt and G.~L.~Alexanderson, Sums of partition sets in generalized Pascal triangles. I,
\emph{Fibonacci Quart.} \textbf{14} (1976), no.~2, 117--125.

\bibitem{Horadam1961}
A.~F.~Horadam, A generalized Fibonacci sequence, \emph{Amer. Math. Monthly} \textbf{68} (1961), 455--459.

\bibitem{Koshy2018}
T.~Koshy, \emph{Fibonacci and Lucas Numbers with Applications}, Volume~1, 2nd Edition, John Wiley \& Sons, Inc., Hoboken, NJ, 2018.

\bibitem{Koshy2014}
T.~Koshy, \emph{{P}ell and {P}ell-{L}ucas Numbers with Applications}, Springer, New York, NY, 2014.

\bibitem{Lucas1878} E.~Lucas, Th\'{e}orie des functions numeriques simplement p\'{e}riodiques, \textit{Amer. J. Math.} \textbf{1} (1878), nos.~2-3, 184--196 and 197--240.

\bibitem{Ribenboim2000}
P.~Ribenboim, \emph{My Numbers, My Friends: Popular Lectures on Number Theory}, Springer-Verlag, New York, 2000.

\bibitem{Ruggles1963}
I.~D.~Ruggles, Some Fibonacci results using Fibonacci-type sequences, \emph{Fibonacci Quart.} \textbf{1} (1963), no.~2, 75--80.

\bibitem{Sloane-OEIS} N.~J.~A. Sloane et al., {\em The On-Line Encyclopedia of Integer Sequences}. Available at \url{https://oeis.org}.

\bibitem{Vajda1989}
S.~Vajda, \emph{{F}ibonacci and {L}ucas Numbers, and the Golden Section}, Ellis Horwood Ltd., Chichester; Halsted Press [John Wiley \& Sons, Inc.], New York, 1989.

\bibitem{CR-B-40-Wall1964}
C.~R.~Wall (solution by J.~Brown), Problem B-40, \emph{Fibonacci Quart.} \textbf{2} (1964), no.~4, 327--328.


\end{thebibliography}
\end{document}